\documentclass[12pt, leqno]{amsart}

\usepackage{amsfonts,amssymb,amsmath,amsthm,latexsym}
\usepackage{url}
\usepackage[utf8]{inputenc}
\usepackage{pgf,tikz}
\usepackage{mathrsfs}
\usetikzlibrary{arrows}
\usepackage{enumerate}
\usepackage{enumerate}
\usepackage[all, knot]{xy}
\usepackage{color}

\newenvironment{red}
{\relax\color{red}}
{\hspace*{.5ex}\relax}

\newcommand{\ber}{\begin{red}}
\newcommand{\er}{\end{red}}

\newenvironment{blue}
{\relax\color{blue}}
{\hspace*{.5ex}\relax}

\newcommand{\beb}{\begin{blue}}
\newcommand{\eb}{\end{blue}}

\newenvironment{green}
{\relax\color{green}}
{\hspace*{.5ex}\relax}

\newcommand{\bev}{\begin{green}}
\newcommand{\ev}{\end{green}}

\newcommand{\bw}{\boldsymbol{w}}

\newcommand{\pw}{\bw_p}
\newcommand{\sw}{\bw_s}

\newcommand{\HH}{\mathcal H}

\newcommand{\PP}{\mathcal{P}}

\newcommand\Tstrut{\rule{0pt}{7.0ex}}
\newcommand\Bstrut{\rule[-3.9ex]{0pt}{0pt}}

\definecolor{ffffff}{rgb}{1.,1.,1.}
\definecolor{qqqqff}{rgb}{0.,0.,1.}
\definecolor{ffqqqq}{rgb}{1.,0.,0.}

\def\diaggdur{\ncnode{d}{s_0}\dbar20pt\ncnode{ }{s_1}
               \cdots \cdots \sbar20pt \kern -2pt \ncnode{ }{s_{n-2}}\sbar20pt \kern -2pt \ncnode{ }{s_{n-1}}}

\def\diaggdeer{{\scriptstyle s}\cnode d\kern -1pt
               \rlap{\kern -10pt \lower 10pt\hbox{$\scriptstyle e+1$}}
	       \raise 16.6pt\hbox{$\ucirc$}
               \kern -28.8pt\lower 12.2pt\hbox{$\lcirc$}
	       \kern -28.8pt\phantom{\dbar 14 pt}
    \kern-3.6pt
    \raise10pt\hbox{$\cnode 2$\rlap{\raise 3pt\hbox{$\kern 2pt\scriptstyle t'_2$}}}
    \raise7pt\hbox{$\diagdown$}
    \kern-17.5pt
    \lower10pt\hbox{$\cnode 2$\rlap{\lower 3pt\hbox{$\kern 2pt\scriptstyle t_2$}}}
    \lower7pt\hbox{$\diagup$}
    \kern -11.6pt\dbar 10pt\ncnode{2}{t_3}\sbar10pt\ncnode{2}{t_4}\cdots
      \ncnode{2}{t_r}
}

\newenvironment{verd}
{\relax\color{magenta}}
{\hspace*{.5ex}\relax}

\newcommand{\bg}{\begin{verd}}
\newcommand{\eg}{\end{verd}}

\def\inv{^{-1}}
\def\ncnode#1#2{{\kern -1pt\mathop\bigcirc\limits_{#2}
                \kern-11pt{\scriptstyle#1}\kern 4pt}}
 \def\nRnode#1#2{{\kern -0.4pt\mathop\Box\limits_{#2}
   \kern-8.6pt{\scriptstyle#1}\kern 2.3pt}}
\def\sbar#1pt{{\vrule width#1pt height3pt depth-2pt}}
\def\dbar#1pt{{\rlap{\vrule width#1pt height2pt depth-1pt}
                 \vrule width#1pt height4pt depth-3pt}}
\newtheorem{thm}{Theorem}[section]
\newtheorem{lem}[thm]{Lemma}
\newtheorem{prop}[thm]{Proposition}
\newtheorem{cor}[thm]{Corollary}

\theoremstyle{definition}
\newtheorem{defn}[thm]{Definition}
\newtheorem{ex}[thm]{Example}

\numberwithin{equation}{section}

\author[J.-Y. Lee, D.-I. Lee and S. Kim]{Jeong-Yup Lee$^1$, Dong-il Lee$^{2,*}$ and Sungsoon Kim$^3$}
\address{Department of Mathematics Education, Catholic Kwandong University\\
Gangwondo 25601, Korea}
\email{jylee@cku.ac.kr}
\thanks{$^1$ This research was supported by
NRF Grant \# 2017078374.}
\address{Department of Mathematics, Seoul Women's University\\
Seoul 01797, Korea}
\email{dilee@swu.ac.kr}
\thanks{$^2$ This research was supported by NRF Grant \# 2018R1D1A1B07044111% and
.}
\thanks{* %This work was done while the author was .
Corresponding author}
\address{LAMFA-CNRS UMR 7352, UPJV, 33 rue St. Leu, 80039 Amiens, France,
(membre assoc. \'equipe des groupes, IMJ-PRG Univ. Paris 7)}
\email{%yjsskim@gmail.com,
sungsoon.kim@u-picardie.fr}
\thanks{$^3$ This author is grateful to KIAS for its hospitality during this work.}

\newcommand{\F}{\mathbb{F}}
\newcommand{\C}{\mathbb{C}}
\newcommand{\Z}{\mathbb{Z}}
\newcommand{\N}{\mathbb{N}}
\newcommand{\G}{Gr\"{o}bner }
\newcommand{\GS}{Gr\"{o}bner-Shirshov }

\newcommand{\tl}{\mathcal{T}}

\begin{document}

\title[Temperley-Lieb Algebras]{Gr\"{o}bner-Shirshov bases for Temperley-Lieb algebras of %the Coxeter group of type $B$ and
the complex reflection group of type $G(d,1,n)$
}

\begin{abstract}
%From the presentation of the Temperley-Lieb algebra of types $B$ and $D$,
We construct a \GS basis of the Temperley-Lieb algebra $\tl(d,n)$ of the complex reflection group $G(d,1,n)$,
inducing
the standard monomials
expressed by the generators $\{ E_i\}$ of $\tl(d,n)$. This result generalizes the one for the Coxeter group of type $B_n$ in \cite{KimSSLeeDI}.
We also give a combinatorial interpretation of the standard monomials of $\tl(d,n)$, relating to the fully commutative elements of the complex reflection group $G(d,1,n)$. In this way, we obtain the dimension formula of $\tl(d,n)$.
\end{abstract}
\maketitle

\section{Introduction}

%\ber EXTRCTED FROM THE PREVIOUS PAPER \er

\bigskip

The Temperley-Lieb algebra appears originally in the context of statistical mechanics \cite{TemperleyLieb},
and later its structure has been studied in connection with knot theory,
where it is known to be a quotient of the Hecke algebra of type $A$ in \cite{Jones1987}.

\medskip

%For any Coxeter group of type $X$, Graham showed that the associated Temperley-Lieb algebra $TL(X)$ has a basis
%consisting of fully commutative elements \cite{Graham}.
%%, following the terminology of Stembridge \cite{Stembridge96}.

Our %method for
approach to understanding the structure of %Coxeter groups
Temperley-Lieb algebras is from the noncommutative \G basis theory,
 or the {\em \GS basis theory} more precisely, which provides a powerful tool for understanding
the structure of (non-)associative algebras and their
representations, especially in computational aspects.
With the ever-growing power of computers,
it is now viewed as a universal engine behind algebraic or symbolic computation.

\medskip

The main interest of the notion of \GS bases stems from Shirshov's Composition Lemma and his algorithm \cite{Sh%, Shirshov_Selected
} %\beb (I need precise pages for this) \eb
for Lie algebras and
independently from Buchberger's algorithm \cite{B65} of computing \G bases for commutative algebras.
In \cite{Bo}, Bokut applied Shirshov's method to associative algebras, and
Bergman mentioned the diamond lemma for ring theory \cite{Be}.
The main idea of
the Composition-Diamond lemma is to establish an algorithm for constructing standard monomials of a quotient algebra by a two-sided ideal generated by a set of relations called \GS basis. Our set of standard monomials in this algorithm is a
minimal set of monomials which are indivisible by any leading monomial of the polynomials in the \GS basis. The details on the \GS basis theory are given in
Section 2.

%On the other hand, the \GS basis theory for representations of associative algebras was developed in \cite{KL} and
%the theory was successfully applied to the representation theory of Lie algebras \cite{KL1, LeeDI, LeeDI1} and Hecke algebras \cite{KLLO, KLLO1}.
\medskip

The \GS bases for Coxeter groups of classical and exceptional types
were completely determined by Bokut, Lee {\it et al.} in \cite{BokutShiao, DenisLee, LeeDI5, LeeDILeeJY, Svechkarenko}.
The cases for Hecke algebras and Temperley-Lieb algebras of type $A$
as well as for Ariki-Koike algebras
were calculated by Lee {\it et al.} in \cite{KLLO, KLLO1, LeeDI2}.

%constructed
\medskip

%\beb
This paper consists of two principal parts as follows~:

\smallskip

\noindent
1) In the first part of this paper, extending the result for type $B_n$ in\cite{KimSSLeeDI}, we construct a \GS basis for the Temperley-Lieb algebra $\tl(d,n)$
of the complex reflection group of type $G(d,1,n)$ and compute the dimension of
$\tl(d,n)$, by enumerating the standard monomails which are in bijection with the fully commutative elements.

The main Theorem goes as follows~:

\begin{thm}[Main theorem \ref{MainThm}]
The algebra $\tl(d,n)$ has a \GS basis $\widehat{R}_{\tl(d,n)}$ with respect to our monomial order $<$ \rm (i.e. degree-lexicographic order with $ E_0< E_1<\cdots<E_{n-1}$\rm )~:
\begin{eqnarray}\label{relation_ctl_GS}
& E_0^d-(d-1)\delta E_0, \nonumber\\
& E_iE_0 - E_0E_i &\mbox{ for } 1< i\le n-1, \nonumber\\
& E_0E_1E_0^kE_1-(k+1)E_0E_1 &\mbox{ for } 1\le k <d,  \nonumber\\
& E_1E_0^kE_1E_0-(k+1)E_1E_0  &\mbox{ for } 1\le k <d,  \nonumber\\
%&E_0E_1E_0^k-E_0^kE_1E_0 &\mbox{ for } 1\le k <d,  \nonumber\\
\widehat{R}_{\tl(d,n)}: & E_{i,1}E_0^kE^{1,j}E_i-E_{i-2,1}E_0^kE^{1,j}E_i &\mbox{ for } i>j+1\ge 1, \nonumber\\
& E_i^2-\delta E_i &\mbox{ for } 1\le i\le n-1, \nonumber\\
& E_iE_j - E_jE_i &\mbox{ for } i> j+1>1, \nonumber\\
& E_{i,j}E_i-E_{i-2,j}E_i &\mbox{ for } i>j>0, \nonumber\\
& E_jE_{i,j}-E_jE_{i,j+2} &\mbox{ for } i>j>0. \nonumber
\end{eqnarray}
%? The corresponding $\widehat{R}_{\tl(d,n)}$-standard monomials are %of the form ...
%exactly the ones in $M_{\tl(d,n)}$.
The cardinality of the set %$M_{\tl(d,n)}$, i.e. the set
of $\widehat{R}_{\tl(d,n)}$-standard monomials is
$$ \dim \tl(d,n)= %(dn-n+d)C_n-d+1
(d-1)(\mathfrak{F}_{n,n-1}(d)-1)+dC_n $$
 where $\mathfrak{F}_{n,k}(x)=\sum_{s=0}^kC(n,s)x^{k-s}$ is the $(n,k)$th Catalan triangle polynomial,
introduced in \cite[\S2.3]{LO16C}.
\end{thm}

We remark here that by specializing $d=2$, we recover the formula for the Temperley-Lieb algebra $\tl(B_n)$ of the $B_n$.

\medskip

\noindent
2) In the second part of this paper, we try to understand some combinatorial aspects on the dimension of the
Temperley-Lieb algebra $\tl(d,n)$.

\smallskip

%In \cite{KleRam2}, Kleshchev and Ram
%defined a class of cuspidal representations of finite types and showed  that every irreducible representation of the KLR algebra (or a quiver Hecke algebra) appears as the head of some induction of these cuspidals. They constructed almost all cuspidal representations whose complete list is done by  Hill, Melvin and Mondragon in  \cite{HillMelMon}. %\\
%
%%In this process of
%Constructing the cuspidal representations, Kleshchev and Ram

\medskip
 In \cite{KleRam, KleRam2}, Kleshchev and Ram
constructed a class of representations called homogeneous representations of the KLR algebra (or a quiver Hecke algebra) and showed that the homogeneous representations can be parametrized by the set of {\it fully commutative elements} of the corresponding Coxeter group. These elements were studied by Fan \cite{Fan96} and Stembridge \cite{Stembridge96} for Coxeter groups and parametrize the bases of the corresponding Temperley-Lieb algebras.

%\medskip

Motivated by the bijective correspondance of the homogeneous representations of KLR algebras
(or a quiver Hecke algebras) with fully commutative elements,  Feinberg and Lee studied the fully commutative elements of the Coxeter groups of types $A$ and $D$ in their papers \cite{FeinLee} and \cite{FeinLee2} and obtained a dimension
formula of the homogeneous representations by using Dyck paths.

\smallskip
\vskip -30pt
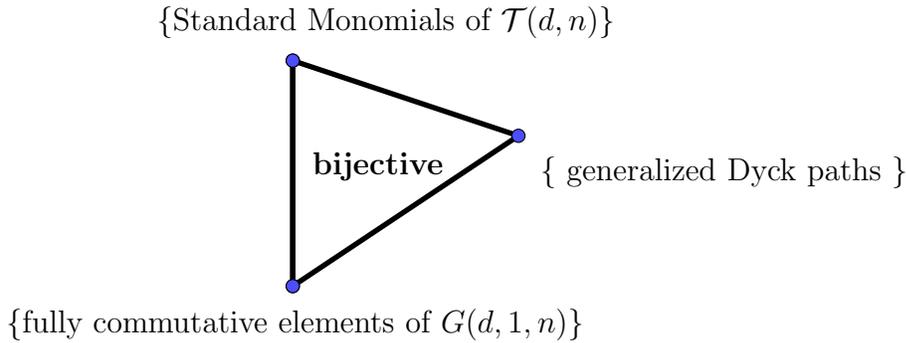
\begin{figure}[h]\label{BijTypeA}
\definecolor{ududff}{rgb}{0.30196078431372547,0.30196078431372547,1.}
\begin{tikzpicture}[line cap=round,line join=round,>=triangle 45,x=1.0cm,y=1.0cm]
\clip(-6.13,0.97) rectangle (12,7);
\draw (-2.95,5.87) node[anchor=north west] {\{Standard Monomials of $\tl(d,n)$\}};
\draw (-4.95,1.85) node[anchor=north west] {\{fully commutative elements of $G(d,1,n)$\}};
\draw (2.15,3.87) node[anchor=north west] {\{ generalized Dyck paths \}};
\draw [line width=2.pt] (-1.,2.) -- (-1.,5.);
\draw [line width=2.pt] (-1.,5.) -- (-1.,2.);
\draw (-0.88,3.93) node[anchor=north west] {\textbf{bijective}};
\draw [line width=2.pt] (-1.,5.)-- (2.,4.);
\draw [line width=2.pt] (2.,4.)-- (-1.,2.);
%\draw (.7,2.93) node[anchor=north west] {\textbf{%For type A, t
%\ber These three sets are bijective.\er}};
\begin{scriptsize}
\draw [fill=ududff] (-1.,5.) circle (2.5pt);
\draw [fill=ududff] (-1.,5.) circle (2.5pt);
\draw [fill=ududff] (2.,4.) circle (2.5pt);
\draw [fill=ududff] (-1.,2.) circle (2.5pt);
\draw [fill=ududff] (-1.,2.) circle (2.5pt);
\end{scriptsize}
\end{tikzpicture}
\caption{These three sets are bijective for $G(d,1,n)$}
\end{figure}

Their combinatorial strategy is as follows~:

For type $A$, we decompose the fully commutative elements into natural subsets according to the lengths of fully commutative elements and show that the fully commutative elements of a given length $k$ can be parametrized by the Dyck paths of semi-length $n$ with the property that
{\it \rm(the sum of peak heights\rm)$-$\rm(the number of peaks\rm)$=k$} using the canonical form of reduced words for fully commutative elements. For type $D$, from the set of fully commutative element written in canonical form, we decompose the fully commutative elements according to the same type of prefixes and call the set with exactly the same prefix {\em a collection}. Then they prove that some collections have the same number of elements by showing that those collections contain the same set of prefixes. We then group those collections together and call the group {\em a Packet}. This decomposition process is called the {\it packet decomposition}.

\medskip

In the article \cite{FKLO}, we generalize the above strategy to the complex
reflection group of type $G(d,1,n)$ for the enumeration of the fully commutative elements of $G(d,1,n)$. The main result on the combinatorial aspects in Section \ref{sectionCOMB} is that there is a bijection between the
standard monomials of $\tl(d,n)$ in the first part of this paper
and the fully commutative elements of $G(d,1,n)$ in the second part of this paper.
In this way, we realize an explicit computation of the dimension of $\tl(d,n)$.

\medskip
Our canonical forms for the reduced elements will be
slightly in different form from the ones in \cite{FKLO}
and so is the packet decomposition process. Though, the fully commutative elements in this paper
are better adapted for the standard monomials induced from a \GS basis.

\section{Preliminaries}

\subsection{\GS basis}
We recall a basic theory of {\em \GS bases} for associative algebras so as to make the paper self-contained. Some properties listed below without proofs are well-known and the readers are invited to see the references cited next to each claim for further detailed explanations.
%
%One remark is that a more advanced theory of {\em \GS pairs} for representation theory
%are given in \cite{KL, KL1} with computational results on Lie algebras and Hecke algebras \cite{KLLO, KLLO1, KLLP, LeeDI, LeeDI1}.
\\

Let $X$ be a set and let $\langle X\rangle$ be the free monoid of associative
words on $X$. We denote the empty word by $1$ and the {\em length} %\beb in the usual sense?\eb
(or {\em degree}) of a word $u$ by $l(u)$.
We define a total-order $<$ on
$\langle X\rangle$, called a {\em monomial order} as follows~; %\\[.1pt]
%\smallskip

\begin{center}{\it if $x<y$ implies $axb<ayb$
for all $a,b\in \langle X\rangle$.} \\%[.1pt]
\end{center}
\smallskip

Fix a monomial order $<$ on $\langle X\rangle$ and let $\mathbb{F}\langle X\rangle$ be the
free associative algebra generated by $X$ over a field $\mathbb{F}$.
Given a nonzero element $p \in \mathbb{F}\langle X\rangle$, we denote by
$\overline{p}$ the monomial (called the {\em leading
monomial}) appearing in $p$, which is maximal under the ordering $<$. Thus $p = \alpha
\overline{p} + \sum \beta _i w_i $ with $\alpha , \beta _i \in
\mathbb{F}$, $ w_i \in \langle X\rangle$, $\alpha \neq 0$ and $w_i <
\overline{p}$ for all $i$. If $\alpha =1$, $p$ is said to be {\em monic}. \smallskip %\\[.1pt]

Let $S$ be a subset of monic elements in
$\mathbb{F}\langle X\rangle$, and let $I$ be the two-sided ideal of $\mathbb{F}\langle X\rangle$
generated by $S$. Then we say
that the algebra $A= \mathbb{F}\langle X\rangle /I$ is {\em defined by $S$}.

\begin{defn}
Given a subset  $S$  of monic elements in
$\mathbb{F}\langle X\rangle$, a monomial $u \in \langle X\rangle$ is said to be {\em $S$-standard} (or {\em $S$-reduced})
if $u$ cannot be expressed as $a\overline{s}b$, that is $u \neq a\overline{s}b$,
for any $s \in S$ and $a, b \in \langle X\rangle$. Otherwise, the monomial $u$ is said to be {\em $S$-reducible}.
\end{defn}

\begin{lem}[\cite{Be, Bo}]\label{division}
Every $p \in \mathbb{F}\langle X\rangle$ can be expressed as
\begin{equation} \label{equ-1}
p = \sum \alpha_i a_is_ib_i + \sum \beta_j
u_j,
\end{equation}
where $\alpha_i, \beta_j \in \mathbb{F}$, $a_i, b_i, u_j \in \langle X\rangle$, $s_i \in S$, $a_i \overline{s_i} b_i \leq
\overline{p}$, $u_j \leq
\overline{p}$ and $u_j$ are $S$-standard.
\end{lem}

\noindent {\it Remark}.
%The proof of the above lemma actually gives us an
%algorithm of writing an element $p$ of $\mathbb{F}\langle X\rangle$ in the form
%(\ref{equ-1}). It may be considered as a {\em division algorithm}.
The term $\sum \beta_j u_j$ in the expression (\ref{equ-1}) is
called a {\em normal form} (or a {\em remainder}) of $p$ with
respect to the subset $S$ (and with respect to the monomial order
$<$). In general, a normal form is not unique.
\\

As an immediate corollary of Lemma \ref{division}, we obtain:

\begin{prop} %{\rm (\cite{KL, KL1})}
The set of $S$-standard monomials spans the algebra
$A=\mathbb{F}\langle X\rangle/I$ defined by the subset $S$, as a vector space over $\mathbb{F}$.
\end{prop}

Let $p$ and $q$ be monic elements in $\F\langle X\rangle$ with leading
monomials $\overline{p}$ and $\overline{q}$. We define the {\em
composition} of $p$ and $q$ as follows.

\begin{defn}
(a) If there exist $a$ and $b$ in $\langle X\rangle$ such that
$\overline{p}a = b\overline{q} = w$ with $l(\overline{p}) > l(b)$,
then we define $(p,q)_w := pa -bq$, called the {\em composition of intersection}.
%Furthermore, if $a=1$, the composition $(p,q)_w$ is called {\em right-justified}.

(b) If there exist $a$ and $b$ in $\langle X\rangle$ such that $a \neq 1$,
$a\overline{p}b=\overline{q}=w$, then we define $(p,q)_{a,b} := apb - q$, called
the {\em composition of inclusion}.
\end{defn}

%\noindent {\it Remark}.
%The composition of inclusion has an ambiguity if we denote it by
%$(p,q)_w$ where $w=a\overline{p}b=\overline{q}$. For example, if
%$p=x_2+x_3$ and $q=x_1x_2^2x_3$, then $(p,q)_w$ may be
%$x_1px_2x_3-q$ or $x_1x_2px_3-q$. So we should specify the monomials $a$ and $b$.
%\\

Let $p, q \in \F\langle X\rangle$ and $w \in \langle X\rangle$. We define the {\em
congruence relation} on $\F\langle X\rangle$ as follows: $p \equiv q
\mod (S; w)$ if and only if $p -q = \sum \alpha_i a_i s_i b_i$, where $\alpha_i \in \mathbb{F}$,
$a_i, b_i \in \langle X\rangle$, $s_i \in S$,  and $a_i
\overline{s_i} b_i < w$.

\begin{defn}
A subset $S$ of monic elements in $\mathbb{F}\langle X\rangle$
is said to be {\em closed under composition} if
\begin{enumerate}
\item[] $(p,q)_w \equiv 0 \mod (S;w)$ and $(p,q)_{a,b} \equiv
0 \mod (S;w)$ for all $p,q \in S$, $a,b \in \langle X\rangle$ whenever the
compositions $(p,q)_w$ and $(p,q)_{a,b}$ are defined.
\end{enumerate}
%If $T=\emptyset$, we will simply say that $S$ is closed under
%composition.
\end{defn}

%The following theorem is the main result of \cite{KL, KL1}, which is
%a generalization of Shirshov's Composition Lemma (for Lie algebras
%and associative algebras) to the representations of associative
%algebras.\\

%\begin{thm}[\cite{KL}]
%Let $(S,T)$ be a pair of subsets of
%monic elements in the free associative algebra $\F\langle X\rangle$
%generated by $X$, let $A=\F\langle X\rangle/J$ be the associative
%algebra defined by $S$, and let $M=A/I$ be the left $A$-module
%defined by $(S,T)$. If $(S,T)$ is closed under composition and the
%image of $p \in \mathbb{F}\langle X\rangle$ is trivial in $M$, then the word
%$\overline{p}$ is $(S, T)$-reducible.
%\end{thm}
%As a corollary, we obtain:

The following theorem is a main tool for our results in the subsequent sections.
%Its representation-theoretic analogue has been proved in \cite{KL, KL1}.

\begin{thm}[Composition Lemma \cite{Be, Bo}]
\label{cor-1}
Let $S$ be a subset of
monic elements in $\F\langle X\rangle$. Then the following conditions
are equivalent\,{\rm :}
\begin{enumerate}
\item [{\rm (a)}] $S$ is closed under composition.
\item [{\rm (b)}] For each $p \in \mathbb{F}\langle X\rangle$, a normal form of $p$ with respect to $S$ is unique.
\item [{\rm (c)}]
The set of $S$-standard monomials forms a linear basis of the
algebra $A=\mathbb{F}\langle X\rangle/I$ defined by $S$.
\end{enumerate}
\end{thm}

\begin{defn}
A subset $S$ of monic elements in $\mathbb{F}\langle X\rangle$ satisfying one of the equivalent conditions in Theorem \ref{cor-1} is called a
a {\em \GS basis} for the algebra $A$ defined by $S$.
\end{defn}

\section{Temperley-Lieb algebras of types $A_{n-1}$ and $B_n$}

\subsection{Temperley-Lieb algebra of type $A_{n-1}$}

First, we review the results on Temperley-Lieb algebras $\tl(A_{n-1})$ $(n\ge 2)$% and Hecke algebras $\h(B_n)$ $(n\ge 2)$
.
Define $\tl(A_{n-1})$ to be the associative algebra over the complex field $\C$, generated by $X=\{E_1,E_2, \ldots, E_{n-1}\}$
with defining relations:
\begin{eqnarray} \label{relation_tlA}
& E_i^2=\delta E_i &\mbox{ for } 1\le i\le n-1, \ \mbox{(idempotent relations)} \nonumber\\
R_{\tl(A_{n-1})}:& E_iE_j = E_jE_i &\mbox{ for } i> j+1, \qquad \mbox{(commutative relations)}\nonumber\\
& E_iE_jE_i=E_i &\mbox{ for } j=i\pm 1, \qquad \mbox{(untwisting relations)} \nonumber
\end{eqnarray}
where $\delta\in\C$ is a parameter.
We call the first and second relations to be the quadratic and commutative relations, respectively.
Our monomial order $<$ is taken to be the degree-lexicographic order with $$
E_1<E_2<\cdots<E_{n-1}.$$ We write $E_{i,j}=E_iE_{i-1}\cdots E_j$ for $i\ge j$ (hence $E_{i,i}=E_i$).
By convention $E_{i,i+1}=1$ for $i\ge 1$.

\begin{prop}{\rm (\cite[Proposition 6.2]{KLLO})}
The Temperley-Lieb algebra $\tl(A_{n-1})$ has a \GS basis $\widehat{R}_{\tl(A_{n-1})}$as follows:
\begin{eqnarray}\label{relation_tlA_GS}
& E_i^2-\delta E_i &\mbox{ for } 1\le i\le n-1, \nonumber\\
\widehat{R}_{\tl(A_{n-1})}:& E_iE_j - E_jE_i &\mbox{ for } i> j+1, \\
& E_{i,j}E_i-E_{i-2,j}E_i &\mbox{ for } i>j, \nonumber\\
& E_jE_{i,j}-E_jE_{i,j+2} &\mbox{ for } i>j. \nonumber
\end{eqnarray}
The corresponding $\widehat{R}_{\tl(A_{n-1})}$-standard monomials are of the form
\begin{equation}\label{monomial_tlA}
E_{i_1,j_1}E_{i_2,j_2}\cdots E_{i_p,j_p}\quad (0\le p\le n-1)
\end{equation}
where $$\begin{aligned}
&1\le i_1<i_2<\cdots <i_p\le n-1,\quad
1\le j_1<j_2<\cdots <j_p\le n-1,\\
&i_1\ge j_1,\ i_2\ge j_2,\ \ldots,\ i_p\ge j_p\quad \text{(the case of $p=0$ is the monomial $1$).}
\end{aligned}$$

We denote the set of $\widehat{R}_{\tl(A_{n-1})}$-standard monomials by $M_{\tl(A_{n-1})}$.
Note that the number of $\widehat{R}_{\tl(A_{n-1})}$-standard monomials equals  the $n^{\text{th}}$ Catalan number, i.e.
\begin{equation}\label{monomial_tlA_count}
|M_{\tl(A_{n-1})}|=\frac{1}{n+1} {2n\choose n} =C_n. \nonumber
\end{equation}
\end{prop}

\noindent {\it Remark}.
There are many combinatorial ways to realize the Catalan number $C_n$, but among those, It is well-known that $C_n$ represents the number of Dyck paths of length $2n$ starting from the point $(0,0)$ ending at $(n,n)$
not passing over the diagonal of the $n \times n$-lattice plane.
\\

\begin{ex}\label{STmonoTypeA}
(1) Note that $|M_{\tl(A_2)}|=C_3=5$. Explicitly,
the $\widehat{R}_{\tl(A_2)}$-standard monomials are as follows:
$$1, E_1, E_{2,1}, E_2, E_1E_2.$$
We will give combinatorial interpretations of this set via Dyck paths and fully
commutative elements in the last section of this article.

(2) Another example with $n=3$~:  we have $|M_{\tl(A_3)}|=C_4=14$. Explicitly,
the $\widehat{R}_{\tl(A_3)}$-standard monomials are as follows:
\begin{eqnarray*}
&1, E_1, E_{2,1}, E_2, E_1E_2, E_{3,1}, E_{3,2}, E_3, \\
&E_1E_{3,2}, E_1E_3, E_{2,1}E_{3,2}, E_{2,1}E_3, E_2E_3, E_1E_2E_3.
\end{eqnarray*}
\end{ex}

%\section{Temperley-Lieb algebras of type $B_n$}

%\bigskip
%\medskip

\subsection{Temperley-Lieb algebra of type $B_n$}

Now we consider the Coxeter diagram for type $B_n$~:

\begin{equation}\label{BnDynkin}
%W(B_n) :
\xymatrix@R=0.5ex@C=4ex{ *{\bullet}<3pt> \ar@{=}[r]_<{0}  &*{\bullet}<3pt>
\ar@{-}[r]_<{1}  &   {} \ar@{.}[r] & *{\bullet}<3pt>
\ar@{-}[r]_>{ \quad n-2} &*{\bullet}<3pt>\ar@{-}[r]_>{
\quad n-1} &*{\bullet}<3pt> } .
\end{equation}
%\ber reverse 0,1,...,n-1\er \beb (done)\eb

Let $W(B_n)$ be the Weyl group with generators $\{s_i\}_{0 \leq i < n}$ and the following defining relations~:

\begin{equation} \label{Eq: WB}
\begin{aligned}
%&\hskip1cm
&\hbox{- quadratic relations} : s_i^2 = 1\,\,\hbox{for } 0\le i\le n-1\\
&\hbox{- braid relations} : \left\{
  \begin{array}{ll}
   s_is_j = s_js_i & \hbox{for } |i-j| > 1\ \ (i,j=1,\ldots,n), \\
   s_is_{i+1}s_i=s_{i+1}s_is_{i+1}& \hbox{for } 1\leq i< n-1,  \\
    s_1s_0s_1s_0=s_0s_1s_0s_1. & \end{array}
 \right. \\[5pt]
\end{aligned}
\end{equation}
Then $W(B_n)$ is the Weyl group of type $B_n$, which is
isomorphic to $(\Z/2\Z)^n\rtimes \mathfrak S_n.$\\[2pt]

Let $\HH := \HH(B_n)$ be the Iwahori-Hecke algebra of type  $B_n$, which is
the associative algebra over $A:=\Z[q,q\inv]$, generated by
$\{T_i\}_{0 \leq i < n}$ with defining relations~:

\begin{equation*} \label{Eq: HB}
\begin{aligned}
&\hbox{- quadratic relations} :  (T_i -q)(T_i+q\inv)=0 \hbox{ for } 0\le i\le n-1, \\
&\hbox{- braid relations} : \left\{
  \begin{array}{ll}
   T_iT_j = T_jT_i & \hbox{for } |i-j| > 1\ \ (i,j=0,\ldots,n-1), \\
   T_iT_{i+1}T_i=T_{i+1}T_iT_{i+1}& \hbox{for } 1\leq i< n-1,  \\
    T_1T_0T_1T_0=T_0T_1T_0T_1 & \end{array}
\right. \\[5pt]
\end{aligned}
\end{equation*}
where $q$ is a parameter and $T_i:=T_{s_i}$, the generator corresponding to the reflection $s_i$.\\[.1pt]

Let $\tl(B_n)$ $(n\ge 2)$ be the Temperley-Lieb algebra of type $B_n$, which is a quotient of the Hecke algebra $\HH(B_n)$. $\tl(B_n)$ is the associative algebra over the complex field $\C$, generated by $X=\{E_0,E_1, \ldots, E_{n-1}\}$
with defining relations~:
\begin{eqnarray}\label{relation_tlB}
& E_i^2=\delta E_i &\mbox{ for } 0\le i\le n-1, \nonumber\\
R_{\tl(B_n)}:& E_iE_j = E_jE_i &\mbox{ for } i> j+1%\quad \mbox{(commutative relations)}
, \\
& E_iE_jE_i=E_i &\mbox{ for } j=i\pm 1,\ i,j>0,  \nonumber\\
& E_iE_jE_iE_j=2E_iE_j &\mbox{ for } \{i,j\}=\{0,1\},  \nonumber
\end{eqnarray}
where $\delta\in\C$ is a parameter.
%Note that $\tl(B_n)$ is a quotient of the Hecke algebra $\h(B_n)$%, and we take the Dynkin diagram reversing that in \cite{KLLO1}.
%We state here the result on Hecke algebras $\h(B_n)$ $(n\ge 2)$.
%Since %the Hecke algebra
%$\h(B_n)$ is the deformed algebra of the group $G(2,1,n)$,
%it is a special case of the Ariki-Koike algebra $\h(G(r,1,n))$, whose \GS basis was calculated in %\cite[Proposition 2.3]{KLLO1}.
%\cite{KLLO1}.
Fix our monomial order $<$ to be the degree-lexicographic order with $$
E_0< E_1<\cdots<E_{n-1}.$$ We write $E_{i,j}=E_iE_{i-1}\cdots E_j$ for $i\ge j\ge 0$, and $E^{i,j}=E_iE_{i+1}\cdots E_j$ for $i\le j$.
By convention, $E_{i,i+1}=1$ and $E^{i+1,i}=1$ for $i\ge 0$. We can easily prove
the following lemma.

\begin{lem}{\rm (\cite[\S4]{KimSSLeeDI})} \label{relation_tlB_Lemma}
The following relations hold in $\tl(B_n)$:
$$E_{i,0}E^{1,j}E_i=E_{i-2,0}E^{1,j}E_i$$ for $i>j+1\ge 1$.
\end{lem}

Let $\widehat{R}_{\tl(B_n)}$ be the collection of defining relations (\ref{relation_tlB})
combined %\ber correction to\\
with (\ref{relation_tlA_GS})
%Add $A_{n-1}$ in the previous section?\er
and the relations in Lemma \ref{relation_tlB_Lemma}.
From this, we denote by $M_{\tl(B_n)}$ the set of $\widehat{R}_{\tl(B_n)}$-standard monomials.
Among the monomials in $M_{\tl(B_n)}$, we consider the monomials
which are not $\widehat{R}_{\tl(A_{n-1})}$-standard.
That is, we take only $\widehat{R}_{\tl(B_n)}$-standard monomials which are not of the form (\ref{monomial_tlA}).
This set is denoted by $M_{\tl(B_n)}^0$.
%$$M_{\tl(B_n)}=M_{\tl(A_{n-1})}\amalg M_{\tl(B_n)}^0$$
Note that each monomial in $M_{\tl(B_n)}^0$ contains $E_0$. We decompose the set $M_{\tl(B_n)}^0$ into two parts as follows~:
$$M_{\tl(B_n)}^0=M_{\tl(B_n)}^{0+}\amalg M_{\tl(B_n)}^{0-}$$
where the monomials in $M_{\tl(B_n)}^{0+}$ are of the form
\begin{equation}\label{monomial_tlB+}
E_0E_{i_1,j_1}E_{i_2,j_2}\cdots E_{i_p,j_p}\quad (0\le p\le n-1) \nonumber
\end{equation}
with $$\begin{aligned}
&1\le i_1<i_2<\cdots <i_p\le n-1,\quad
0\le j_1\le j_2\le \cdots \le j_p\le n-1,\\
&i_1\ge j_1,\ i_2\ge j_2,\ \ldots,\ i_p\ge j_p,\ \mbox{ and }\\
&j_k>0\ (1\le k<p)\ \mbox{ implies } j_k<j_{k+1}\end{aligned}$$
(the case of $p=0$ is the monomial $E_0$),
and the monomials in $M_{\tl(B_n)}^{0-}$ are of the form
\begin{equation}
E_{i_1,j_1}'E_{i_2,j_2}\cdots E_{i_p,j_p}\quad (1\le p\le n-1) \nonumber
\end{equation}
with $$E_{i,j}'=E_{i,0}E^{1,j}$$ and the same restriction on $i$'s and $j$'s as above.
%$$\begin{aligned}
%&1\le i_1<i_2<\cdots <i_p\le n-1,\quad
%0\le j_1\le j_2\le \cdots \le j_p\le n-1,\\
%&i_1\ge j_1,\ i_2\ge j_2,\ \ldots,\ i_p\ge j_p, \mbox{ and }\\
%&j_q>0 \mbox{ implies } j_q<j_{q+1} \mbox{ for } 1\le q<p.\end{aligned}$$
%(the case of $p=0$ is the monomial $1$).
It can be easily checked that
$M_{\tl(B_n)}^0$ is the set of $\widehat{R}_{\tl(B_n)}$-standard monomials which are not $\widehat{R}_{\tl(A_{n-1})}$-standard.
\medskip

Counting the number of elements in $M_{\tl(B_n)}^0$, we obtain the following theorem.

\begin{thm}{\rm (\cite[\S4]{KimSSLeeDI})} %\label{MainThm}
The algebra $\tl(B_n)$ has a \GS basis $\widehat{R}_{\tl(B_n)}$ with respect to our monomial order $<$:
\begin{eqnarray}\label{relation_tlB_GS}
& E_i^2-\delta E_i &\mbox{ for } 0\le i\le n-1, \nonumber\\
& E_iE_j - E_jE_i &\mbox{ for } i> j+1, \nonumber\\
\widehat{R}_{\tl(B_n)}: & E_{i,j}E_i-E_{i-2,j}E_i &\mbox{ for } i>j>0, \nonumber\\
& E_jE_{i,j}-E_jE_{i,j+2} &\mbox{ for } i>j>0. \nonumber\\
& E_iE_jE_iE_j-2E_iE_j &\mbox{ for } \{i,j\}=\{0,1\},  \nonumber\\
& E_{i,0}E^{1,j}E_i-E_{i-2,0}E^{1,j}E_i &\mbox{ for } i>j+1\ge 1. \nonumber
\end{eqnarray}
The cardinality of the set $M_{\tl(B_n)}$, i.e. the set of  $\widehat{R}_{\tl(B_n)}$-standard monomials, is
$$\dim \tl(B_n)=(n+2)C_n-1.$$
\end{thm}

\begin{ex}\label{STmonoTypeB}

For $n=3$, we have $|M_{\tl(B_3)}|= (3+2)C_3 -1=24$.

We enumerate the standard monomials in $M_{\tl(B_3)}^0= M_{\tl(B_3)}\setminus M_{\tl(A_3)}$, that is, the standard monomials containing $E_0$, of cardinality $24-5=19$ as follows (using the notation $E_{i,j}':=E_{i,0}E^{1,j}$)~:

%We enumerate the $\widehat{R}_{\tl(B_3)}$-standard monomials containing $E_0$:
\begin{eqnarray*}
&E_0,\ E_0E_{1,0}, E_{1,0},\ E_0E_1, E_1'=E_1E_0E_1,\ E_0E_{2,0}, E_{2,0},\ E_0E_{2,1},\\
&E_{2,1}'=E_2E_1E_0E_1, E_0E_2, E_2'=E_2E_1E_0E_1E_2,\  E_0E_{1,0}E_{2,0}, E_{1,0}E_{2,0},\\
& E_0E_{1,0}E_{2,1}, E_{1,0}E_{2,1}, E_0E_{1,0}E_2, E_{1,0}E_2, E_0E_1E_2, E_1'E_2=E_1E_0E_1E_2.
\end{eqnarray*}

%(2) The product of two $\widehat{R}_{\tl(B_3)}$-standard monomials is a scalar multiple of a standard monomial.
%%
%%We have the multiplication table between the $\widehat{R}_{\tl(B_3)}$-standard monomials.
%For instance, we multiply $E_0E_{1,0}E_{2,0}$ by $E_2$ from the left:
%$$
%E_2(E_0E_{1,0}E_{2,0})=E_0E_{2,0}E_{2,0}=E_0 E_0E_2 E_{1,0}=\delta E_0 E_{2,0}.
%$$
%Note that the second equality comes from the Lemma \ref{relation_tlB_Lemma}.
\end{ex}

%\ber Can we give the bijective correspondance between the fully commutative elements and the standard monomials for each type of $A_{n-1}$ and $B_n$?
%Also, if we can give one more bijective correspondance between the Dyck paths
%and one of the above sets, that will be great. Probably it can be easily explained because $C(n,k)$ corresponds to the Dyck paths passin through the point $(n,k)$
%where $k< n$.\er

%\newpage
\bigskip
\section{Gr\"{o}bner-Shirshov bases for Temperley-Lieb algebras of the complex reflection group of type $G(d,1,n)$}

In this section, we try to generalize the result of the previous section to the
complex reflection group of type $G(d,1,n)$, where %$d,n\in \N_{\geq 2}.$
$d\ge 2$ and $n\ge 2$.
Let $W_n := G(d,1,n)$ be the wreath product $(\Z/d\Z)^n\rtimes \mathfrak S_n$ of the cyclic group $\Z/d\Z$ and the symmetric group $\mathfrak S_n$, with the following diagram~:

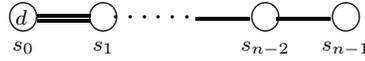
$$\diaggdur$$
%\ber reverse 0,1,...,n-1\er \beb done\eb

In other words, $W_n$ is the group with generators $\{s_i\}_{0\leq i < n}$ and the following defining relations~:

\begin{equation*} \label{Eq: WGd1n}
\begin{aligned}
&%\hskip1cm
s_i^2= 1\,\,\hbox{for } 1\le i\le n-1\,\,\quad \hbox{and}\,\, s_0^d=1,\\
&\hbox{braid relations} : \left\{
  \begin{array}{ll}
   s_is_j = s_js_i & \hbox{if } |i-j| > 1\ \ (i,j=0,\ldots,n-1), \\
   s_is_{i+1}s_i=s_{i+1}s_is_{i+1}& \hbox{if } 1\leq i< n-1,  \\
   s_1s_0s_1s_0=s_0s_1s_0s_1 & \end{array}
 \right. \\[5pt]
\end{aligned}
\end{equation*}

Let $\HH_n := \HH(W_n)$ be the Ariki-Koike algebra of type  $G(d,1,n)$, %\ber
cyclotomic Hecke algebra %\er
which is the associative algebra over $A:=\Z[\zeta,q,q\inv]$, generated by
$\{T_i\}_{0\leq i < n}$ with defining relations~:

\begin{equation*} \label{Eq: HB}
\begin{aligned}
& (T_i -q)(T_i+q\inv)=0 \hbox{ for } 0\le i\le n-1, \\
& (T_0-q^{m_0})(T_0-q^{m_1}\zeta)(T_0-q^{m_2}\zeta^2)\cdots (T_0-q^{m_{d-1}}\zeta^{d-1})=0, \\
&\hbox{braid relations} : \left\{
  \begin{array}{ll}
   T_iT_j = T_jT_i & \hbox{for } |i-j| > 1\ \ (i,j=0,\ldots,n-1), \\
   T_iT_{i+1}T_i=T_{i+1}T_iT_{i+1}& \hbox{for } 1\leq i< n-1,  \\
    T_1T_0T_1T_0=T_0T_1T_0T_1 & \end{array}
\right. \\[5pt]
\end{aligned}
\end{equation*}
where $\zeta$ is a primitive $d^{th}$ root of unity and $T_i:=T_{s_i}$, the generator corresponding to the reflection $s_i$, and $m_i \in \N$.\\[.1pt]

Similarly, we define
the Temperley-Lieb algebra of %type
$G(d,1,n)$ $(d,n\ge 2)$, denoted by $\tl(d,n)$, as a quotient of the Ariki-Koike algebra of $G(d,1,n)$ $(d,n\ge 2)$.

\begin{defn}
The Temperley-Lieb algebra $\tl(d,n)$ is the associative algebra over the complex field $\C$, generated by $X=\{E_0,E_1, \ldots, E_{n-1}\}$
with defining relations:
\begin{eqnarray}\label{relation_ctl}
& E_i^2=\delta E_i &\mbox{ for } 1\le i\le n-1, \nonumber\\
&  E_0^d=(d-1)\delta E_0,  \nonumber\\
R_{\tl(d,n)}
:& E_iE_j = E_jE_i &\mbox{ for } i> j+1%\quad \mbox{(commutative relations)}
, \\
& E_iE_jE_i=E_i &\mbox{ for } j=i\pm 1,\ i,j>0,  \nonumber\\
%& E_iE_jE_iE_j=2E_iE_j &\mbox{ for } \{i,j\}=\{0,1\},  \nonumber
& E_0E_1E_0^kE_1=(k+1)E_0E_1&\mbox{ for } 1\le k <d,  \nonumber\\
& E_1E_0^kE_1E_0=(k+1)E_1E_0&\mbox{ for } 1\le k <d,  \nonumber
\end{eqnarray}
where $\delta%, \delta_0
\in\C$ is a parameter.
\end{defn}

Fix the same monomial order $<$ as in the previous section, i.e. the degree-lexicographic order with $ E_0< E_1<\cdots<E_{n-1}$.

\begin{lem} \label{relation_ctl_Lemma}
The following relations hold in $\tl(d,n)$;
%\begin{enumerate}[{\rm (a)}]\item
%, we have
$$E_{i,1}E_0^kE^{1,j}E_i=E_{i-2,1}E_0^kE^{1,j}E_i$$
for $i>j+1\ge 1$ and $1\le k<d$.
%\item For $1\le k<d$, we have $$E_0E_1E_0^k=E_0^kE_1E_0.$$
%\end{enumerate}
\end{lem}

\begin{proof} %(a)
Since $2\le i\le n-1$ and $0\le j\le i-2$, we calculate that $$
E_{i,1}E_0^kE^{1,j}E_i=(E_iE_{i-1}E_i)E_{i-2,1}E_0^kE^{1,j}=E_iE_{i-2,1}E_0^kE^{1,j}=E_{i-2,1}E_0^kE^{1,j}E_i$$
by the commutative relations and $E_iE_{i-1}E_i=E_i$.
%
%(b) The relation $$\textstyle E_0E_1E_0^k-E_0^kE_1E_0=E_0(\frac{1}{2}E_1E_0^kE_1E_0)-(\frac{1}{2}E_0E_1E_0^kE_1)E_0=0$$ is deduced.
\end{proof}

The set of defining relations (\ref{relation_ctl})
combined with (\ref{relation_tlA_GS}) and the relations in Lemma \ref{relation_ctl_Lemma}
is denoted by $\widehat{R}_{\tl(d,n)}$.
%
%$E_0^kE_1E_0E_1=E_0^{k-1}(2E_0E_1)=2E_0^kE_1$. \\ $E_1E_0E_1E_0^k=(2E_1E_0)E_0^{k-1}=2E_1E_0^k$.\\
%\ber $E_0E_1E_0^kE_1=?$ \er \\ \ber $E_1E_0^kE_1E_0=?$ \er \\
%
We enumerate $\widehat{R}_{\tl(d,n)}$-standard monomials containing $E_0$ considering the following two types of standard monomials.

The first type of standard monomials is~:
\begin{equation}\label{monomial_ctl0+}
%E_0^kE_{i_1,j_1}E_{i_2,j_2}\cdots E_{i_p,j_p}\quad (0\le p\le n-1)
%
%E_0^{k_0}
(E_{i_1,1}E_0^{k_1})(E_{i_2,1}E_0^{k_2})\cdots (E_{i_q,1}E_0^{k_q})E_{i_{q+1},j_{q+1}}\cdots E_{i_p,j_p}
\quad (1\le p\le n-1, 1\le q\le p)
\end{equation}
with $$\begin{aligned}
&%0\le k_0<d,\
1\le k_1,k_2,\ldots,k_q<d,\quad 0\le i_1<i_2<\cdots<i_q<i_{q+1}<\cdots <i_p\le n-1,\\%\quad
&1\le j_{q+1}< j_{q+2}< \cdots < j_p\le n-1,\quad i_{q+1}\ge j_{q+1},\ i_{q+2}\ge j_{q+2},\ \ldots,\ i_p\ge j_p.
%,\ \mbox{ and }\\ &j_k>0\ (1\le k<p)\ \mbox{ implies } j_k<j_{k+1}.
\end{aligned}$$
%(the case of $p=0$ is the monomial $E_0^{k_0}$)

In a similar way as we did in the proof of \cite[Theorem 4.2]{KimSSLeeDI},
to count the number of standard monomials of the above form (\ref{monomial_ctl0+}),
we associate those standard monomials bijectively to the paths defined below.\\

%\ber
%{\bf Path realization}:
To each monomial $$
%E_0^kE_{i_1,0}\cdots E_{i_q,0}E_{i_{q+1},j_{q+1}}\cdots E_{i_p,j_p}$$ with $j_{q+1}>0$,
 (E_{i_1,1}E_0^{k_1})\cdots (E_{i_q,1}E_0^{k_q})E_{i_{q+1},j_{q+1}}\cdots E_{i_p,j_p},$$
we associate
a unique path, which we call a {\it $G(d,1,n)$-Dyck path},  %(often called Dyck path)
\begin{equation*} %\label{Gd1n-DyckPath}
%(0,0)^{k_0}\rightarrow
(i_1,0)^{k_1}\rightarrow \cdots \rightarrow (i_q,0)^{k_q} \rightarrow (i_{q+1},j_{q+1})\rightarrow \cdots \rightarrow (i_p,j_p)\rightarrow (n,n).
\end{equation*}
Here,
a path consists of moves to the east or to the north, {\em not above the diagonal} in the lattice plane.
The move from $(i,j)$ to $(i',j')$ ($i<i'$ and $j<j'$) is a concatenation of eastern moves followed by northern moves.
As an example, the monomial $E_0^2E_{1,0}E_{2,1}$ in $\tl(3,3)$ corresponds to
the figure below $$
\textbf{A}(0,0)^2\rightarrow \textbf{B}(1,0) {\boldsymbol{\rightarrow}}\textbf{C}(2,1)\rightarrow (3,3).$$
%\medskip
%\er
\vskip-1cm
\begin{table}[h]
\begin{tikzpicture}[line cap=round,line join=round,>=triangle 45,x=1.0cm,y=1.0cm]
\clip(-1.98,-1.00) rectangle (5,3.4);
\draw [line width=1.2pt] (0.,0.)-- (1.02,0.005);
\draw (0.7,-0.155) node[anchor=north west] {(1,0)};
\draw (1.7,1.7) node[anchor=north west] {(2,1)};
\draw (3.,3.2) node[anchor=north west] {(3,3)};
\draw [line width=1.2pt] (1.02,0.005)-- (2.,0.005);
\draw [line width=1.2pt] (2.,0.005)-- (2.,1.005);
\draw [line width=1.2pt] (2.,1.005)-- (3.,1.025);
\draw [line width=1.2pt] (3.,1.025)-- (2.998447893569847,3.0342793791574283);
\draw [color=ffffff](1.52,2.8) node[anchor=north west] {C};
\draw (2.,1.) node[anchor=north west] {\textbf{C}};
\draw (-0.7, 0.479) node[anchor=north west] {\textbf{A}};
\draw (0.9,0.585) node[anchor=north west] {\textbf{B}};
\draw [line width=1.2pt,dotted] (0.,0.)-- (2.998447893569847,3.0342793791574283);
\draw (-0.7,-0.155) node[anchor=north west] {$\mathbf{(0,0)^2}$};
\begin{scriptsize}
\draw [fill=qqqqff] (0.,0.) ++(-2.5pt,0 pt) -- ++(2.5pt,2.5pt)--++(2.5pt,-2.5pt)--++(-2.5pt,-2.5pt)--++(-2.5pt,2.5pt);
\draw [fill=qqqqff] (2.998447893569847,3.0342793791574283) circle (2.5pt);
\draw [fill=qqqqff] (1.02,0.005) circle (2.5pt);
\draw [fill=qqqqff] (2.,1.005) circle (2.5pt);
\draw [fill=qqqqff] (2.,0.005) circle (2.5pt);
\draw [fill=qqqqff] (3.,1.025) circle (2.5pt);
\end{scriptsize}
\end{tikzpicture}
\caption{The path corresponding to $E_0^2E_{1,0}E_{2,1}$ in $\tl(3,3)$}\label{DyckPathGdn}
\end{table}

\noindent
We note that in the move from $B(1,0)$ to $C(2,1)$,  there is a concatenation of an eastern move followed by a northern move as well as in the move from $C(2,1)$ to
the plot $(3,3)$ as shown in Table~\ref{DyckPathGdn}.
%\ber Taking the monomial notation $E_{i,0}^{(k)}$ with $i\ge 0$ would be better %instead of $E_{i,1}E_0^k$ in (\ref{monomial_ctl0+})
%to clearly show the correspondence between the standard monomials and
%$G(d,1,n)$-Dyck paths.
%\er

\medskip
The second type of standard monomials is:
\begin{equation}\label{monomial_ctl0-}
(E_{i_1,1}E_0^kE^{1,j_1})E_{i_2,j_2}E_{i_3,j_3}\cdots E_{i_p,j_p}\quad (1\le p\le n-1)
\end{equation}
with %the same restriction on $i$'s and $j$'s as above.
$$\begin{aligned}
&1\le k<d,\quad 1\le i_1<i_2<\cdots <i_p\le n-1,\\
&1\le j_1< j_2< \cdots <j_p\le n-1,\quad i_1\ge j_1,\ i_2\ge j_2,\ \ldots,\ i_p\ge j_p%,\ \mbox{ and }\\
%&j_k>0\ (1\le k<p)\ \mbox{ implies } j_k<j_{k+1}
.\end{aligned}$$

%Consider the monomials, for $2\le p\le n-1$ and $1\le q< p$,
%\begin{equation}\label{monomial_ctl01}
%(E_{i_1,1}E_0^{k_1})\cdots (E_{i_q,1}E_0^{k_q}) (E_{i_{q+1},1}E_0^{k_{q+1}}E^{1,j_{q+1}})E_{i_{q+2},j_{q+2}}\cdots E_{i_p,j_p}
%%(2\le p\le n-1,\ 1\le q< p)
%\end{equation} with $$\begin{aligned}
%&1\le k_1,\ldots,k_q<d,\quad 2\le k_{q+1}<d, \quad 1\le i_1<i_2<\cdots <i_p\le n-1,\\
%&1\le j_{q+1}< \cdots <j_p\le n-1,\quad i_{q+1}\ge j_{q+1},\ \ldots,\ i_p\ge j_p.%,\ \mbox{ and }\\
%\end{aligned}$$ We count the number of monomials according to $$i_q=0,1,2,\ldots,n-2%,n-1
%$$ to have $$\begin{aligned}
%&\textstyle (d-1)(d-2)(C_n-1)+(d-1)(d-2)\left(\frac{3}{n+1}{2n-2\choose n}-1\right)d \\
%&\textstyle +(d-1)(d-2)\left(\frac{4}{n+1}{2n-3\choose n}-1\right)d^2 +\cdots +(d-1)(d-2)\left(\frac{n}{n+1}{n+1\choose n}-1\right)d^{n-2} %%+(d-1)(d-2)(1-1)d^{n-1}
%\\ &\textstyle =(d-1)(d-2)\sum_{s=0}^{n-1}(C(n,s)-1)d^{n-1-s}\\
%&\textstyle =(d-2)\left((d-1)\F_{n,n-1}(d)-d^n+1\right).\end{aligned}$$
%\
%
%?? $E_1E_0^kE_1E_0=?$

\medskip
Now we let $M_{\tl(d,n)}$ be the set of standard monomials of the forms (\ref{monomial_ctl0+}) and (\ref{monomial_ctl0-}) combined with (\ref{monomial_tlA}).

\medskip
The following is our main theorem.

\begin{thm} \label{MainThm}
The algebra $\tl(d,n)$ has a \GS basis $\widehat{R}_{\tl(d,n)}$ with respect to our monomial order $<$ \rm (i.e. degree-lexicographic order with $ E_0< E_1<\cdots<E_{n-1}$\rm )~:
\begin{eqnarray}\label{relation_ctl_GS}
&  E_0^d-(d-1)\delta E_0, \nonumber\\
& E_iE_0 - E_0E_i &\mbox{ for } 1< i\le n-1, \nonumber\\
& E_0E_1E_0^kE_1-(k+1)E_0E_1  &\mbox{ for } 1\le k <d,  \nonumber\\
& E_1E_0^kE_1E_0-(k+1)E_1E_0 &\mbox{ for } 1\le k <d,  \nonumber\\
%&E_0E_1E_0^k-E_0^kE_1E_0 &\mbox{ for } 1\le k <d,  \nonumber\\
\widehat{R}_{\tl(d,n)}: & E_{i,1}E_0^kE^{1,j}E_i-E_{i-2,1}E_0^kE^{1,j}E_i &\mbox{ for } i>j+1\ge 1, \nonumber\\
& E_i^2-\delta E_i &\mbox{ for } 1\le i\le n-1, \nonumber\\
& E_iE_j - E_jE_i &\mbox{ for } i> j+1>1, \nonumber\\
& E_{i,j}E_i-E_{i-2,j}E_i &\mbox{ for } i>j>0, \nonumber\\
& E_jE_{i,j}-E_jE_{i,j+2} &\mbox{ for } i>j>0. \nonumber
\end{eqnarray}
%? The corresponding $\widehat{R}_{\tl(d,n)}$-standard monomials are %of the form ...
%exactly the ones in $M_{\tl(d,n)}$.
The cardinality of the set $M_{\tl(d,n)}$, i.e. the set of $\widehat{R}_{\tl(d,n)}$-standard monomials, is
$$ \dim \tl(d,n)= %(dn-n+d)C_n-d+1
(d-1)(\mathfrak{F}_{n,n-1}(d)-1)+dC_n $$
 where $\mathfrak{F}_{n,k}(x)=\sum_{s=0}^kC(n,s)x^{k-s}$ is the $(n,k)$th Catalan triangle polynomial,
introduced in \cite[\S2.3]{LO16C}.
\end{thm}

\begin{proof}
%We can easily check that all $\widehat{R}_{\tl(d,n)}$-standard monomials are the ones in $M_{\tl(d,n)}$.\\[.1pt]

Since the set of relations (\ref{relation_tlA_GS}) which do not contain $E_0$ is already closed under composition,
we have only to check for the relations containing $E_0$, to show that $\widehat{R}_{\tl(d,n)}$ is closed under composition.

\begin{enumerate}[{\rm (I)}]
\item First, consider the compositions between $E_0^d-(d-1)\delta E_0$ and another relation.

\begin{enumerate}[{\rm 1)}]
\item $%\begin{aligned}
(E_iE_0 - E_0E_i)E_0^{d-1}-E_i\left(E_0^d-(d-1)\delta E_0\right)\\=-E_0E_iE_0^{d-1}+(d-1)\delta E_iE_0
=-E_i\left(E_0^d-(d-1)\delta E_0\right)=0$.

\item  $\left(E_0^d-(d-1)\delta E_0\right)E_1E_0^kE_1-E_0^{d-1}\left(E_0E_1E_0^kE_1-(k+1)E_0E_1\right)\\
=-(d-1)\delta E_0E_1E_0^kE_1+(k+1)E_0^dE_1\\
= -(d-1)\delta (k+1)E_0E_1+(k+1)(d-1)\delta E_0E_1 %=E_0^{k-1}E_1(-\delta E_0+E_0^d)
=0$.

\item $\left(E_1E_0^kE_1E_0-(k+1)E_1E_0\right)E_0^{d-1}-E_1E_0^kE_1\left(E_0^d-(d-1)\delta E_0\right)\\
=-(k+1)E_1E_0^d+(d-1)\delta E_1E_0^kE_1E_0\\
=-(k+1)(d-1)\delta E_1E_0+(d-1)\delta (k+1)E_1E_0 %=-2E_1E_0^{k-1}(E_0^d-\delta E_0)
=0$.

%\item $(E_0^d-\delta E_0)E_1E_0^k-E_0^{d-1}(E_0E_1E_0^k-E_0^kE_1E_0)=-\delta E_0E_1E_0^k+E_0^{d-1+k}E_1E_0\\
%=-\delta E_0E_1E_0^k+\delta E_0^{k}E_1E_0=-\delta (E_0E_1E_0^k-E_0^{k}E_1E_0)=0$.
\end{enumerate}

\medskip
\item Next, we take the relation $E_iE_0 - E_0E_i$ and another relation.
\begin{enumerate}[{\rm 1)}]
\item $(E_i^2-\delta E_i)E_0-E_i(E_iE_0 - E_0E_i)=-\delta E_iE_0+E_iE_0E_i
=E_0(-\delta E_0+E_i^2)=0$.

\item $(E_iE_0 - E_0E_i)E_1E_0^kE_1- E_i\left(E_0E_1E_0^kE_1-(k+1)E_0E_1\right)\\
    =-E_0E_iE_1E_0^kE_1+(k+1)E_iE_0E_1
=-E_i\left(E_0E_1E_0^kE_1-(k+1)E_0E_1\right)=0$.

%\item $(E_iE_0 - E_0E_i)E_1E_0^k-E_i(E_0E_1E_0^k - E_0^kE_1E_0)=- E_0E_iE_1E_0^k+E_iE_0^kE_1E_0\\
%=- E_i(E_0E_1E_0^k-E_0^kE_1E_0)=0$.

\item $(E_{i,1}E_0^kE^{1,j}E_i-E_{i-2,1}E_0^kE^{1,j}E_i)E_0- E_{i,1}E_0^kE^{1,j}(E_iE_0-E_0E_i) \\ =-E_{i-2,1}E_0^kE^{1,j}E_iE_0+E_{i,1}E_0^kE^{1,j}E_0E_i\\
    =(-E_{i-2,1}E_0^kE^{1,j}E_i+E_{i,1}E_0^kE^{1,j}E_i)E_0=0$.

\item $(E_iE_j - E_jE_i)E_0-E_i(E_jE_0 - E_0E_j)=-E_jE_iE_0+E_iE_0E_j=0$.
\item $(E_{i,j}E_i - E_{i-2,j}E_i)E_0-E_{i,j}(E_iE_0 - E_0E_i)\\=-E_{i-2}E_iE_0+E_{i,j}E_0E_i
=(-E_{i-2}E_i+E_{i,j}E_i)E_0=0$.
\item $(E_jE_{i,j} - E_jE_{i,j+2})E_0-E_jE_{i,j+1}(E_jE_0 - E_0E_j)\\=-E_jE_{i,j+2}E_0+E_jE_{i,j+1}E_0E_j
=(-E_jE_{i,j+2}+E_jE_{i,j})E_0=0$.
\end{enumerate}

\medskip
\item Calculate for the relation $E_0E_1E_0^kE_1-(k+1)E_0E_1$ and another relation.
\begin{enumerate}[{\rm 1)}]
\item $\left(E_0E_1E_0^kE_1-(k+1)E_0E_1\right)E_0^kE_1-E_0E_1E_0^{k-1}\left(E_0E_1E_0^kE_1-(k+1)E_0E_1\right)\\
=-(k+1)E_0E_1E_0^kE_1+(k+1)E_0E_1E_0^kE_1 %=-2E_0^{k-1}(2E_0^kE_1)+2(2E_0^{2k-1}E_1)
=0$.
\item $\left(E_0E_1E_0^kE_1-(k+1)E_0E_1\right)E_0-E_0\left(E_1E_0^kE_1E_0-(k+1)E_1E_0\right)\\
    =-(k+1)E_0E_1E_0+(k+1)E_0E_1E_0=0$.
\item $\left(E_0E_1E_0^kE_1-(k+1)E_0E_1\right)E_0^kE_1E_0-E_0E_1E_0^k\left(E_1E_0^kE_1E_0-(k+1)E_1E_0\right)\\
=-(k+1)E_0E_1E_0^kE_1E_0+(k+1)E_0E_1E_0^kE_1E_0 %=-2E_0^k(2E_1E_0^k)+2(2E_0^kE_1)E_0^k
=0$.
\item $\left(E_1E_0^kE_1E_0-(k+1)E_1E_0\right)E_0^{k-1}E_1- E_1E_0^{k-1}\left(E_0E_1E_0^kE_1-(k+1)E_0E_1\right)\\
=-(k+1)E_1E_0^kE_1+(k+1)E_1E_0^kE_1=0$.
\item $\left(E_1E_0^kE_1E_0-(k+1)E_1E_0\right)E_1E_0^kE_1- E_1E_0^kE_1\left(E_0E_1E_0^kE_1-(k+1)E_0E_1\right)\\
=-(k+1)E_1E_0E_1E_0^kE_1+(k+1)E_1E_0^kE_1E_0E_1\\=-(k+1)E_1(k+1)E_0E_1+(k+1)^2E_1E_0E_1
=0$.

%\item $(E_0E_1E_0^kE_1-2E_0^kE_1)-(E_0E_1E_0^k-E_0^kE_1E_0)E_1=-2E_0^kE_1+E_0^kE_1E_0E_1=0$.
%\item $(E_0E_1E_0^kE_1-2E_0^kE_1)E_0^k-E_0E_1E_0^{k-1}(E_0E_1E_0^k-E_0^kE_1E_0)\\
%=-2E_0^kE_1E_0^k+E_0E_1E_0^{2k-1}E_1E_0=-2E_0^{k-1}(E_0^kE_1E_0)+(2E_0^{2k-1}E_1)E_0=0$.
%\item $(E_0E_1E_0^k-E_0^kE_1E_0)E_1E_0^kE_1-E_0E_1E_0^{k-1}\left(E_0E_1E_0^kE_1-(k+1)E_0E_1\right)\\
%=-E_0^kE_1E_0E_1E_0^kE_1+2E_0E_1E_0^{2k-1}E_1=-E_0^kE_1(2E_0^kE_1)+2(2E_0^{2k-1}E_1)\\
%=-E_0^{k-1}(4E_0^kE_1)+4E_0^{2k-1}E_1=0$.

\item $\left(E_0E_1E_0^kE_1-(k+1)E_0E_1\right)E_1-E_0E_1E_0^k(E_1^2-\delta E_1)\\=-(k+1)E_0E_1^2+\delta E_0E_1E_0^kE_1
=-(k+1)\delta E_0E_1+\delta (k+1)E_0E_1 %=\delta(-2E_0^kE_1+E_0E_1E_0^kE_1)
=0$.
\item $\left(E_0E_1E_0^kE_1-(k+1)E_0E_1\right)E_{i,1}-E_0E_1E_0^k(E_1E_{i,1}-E_1E_{i,3})\\
    =-(k+1)E_0E_1E_{i,1}+E_0E_1E_0^kE_1E_{i,3}\\
=-(k+1)E_0E_1E_{i,1}+(k+1)E_0E_1E_{i,3}=-(k+1)E_0(E_1E_{i,1}-E_1E_{i,3})=0$.
\end{enumerate}

\medskip
\item Check for the relation $E_1E_0^kE_1E_0-(k+1)E_1E_0$ and another relation.
\begin{enumerate}[{\rm 1)}]
\item $\left(E_1E_0^kE_1E_0-(k+1)E_1E_0\right)E_0^{k-1}E_1E_0-E_1E_0^k\left(E_1E_0^kE_1E_0-(k+1)E_1E_0\right)\\
=-(k+1)E_1E_0^kE_1E_0+(k+1)E_1E_0^kE_1E_0 %=-2(2E_1E_0^{2k-1})+2(2E_1E_0^k)E_0^{k-1}
=0$.

%\item $(E_1E_0^kE_1E_0-2E_1E_0^k)E_0^{k-1} -E_1E_0^{k-1}(E_0E_1E_0^k-E_0^kE_1E_0)\\
%=-2E_1E_0^{2k-1}+E_1E_0^{2k-1}E_1E_0=0$.
%\item $(E_1E_0^kE_1E_0-2E_1E_0^k)E_1E_0^k -E_1E_0^kE_1(E_0E_1E_0^k-E_0^kE_1E_0)\\
%=-2E_1E_0^kE_1E_0^k+E_1E_0^kE_1E_0^kE_1E_0=E_1E_0^k(-2E_1E_0^k+E_1E_0^kE_1E_0)=0$.
%\item $(E_0E_1E_0^k-E_0^kE_1E_0)E_1E_0 -E_0\left(E_1E_0^kE_1E_0-(k+1)E_1E_0\right)\\
%=-E_0^kE_1E_0E_1E_0+2E_0E_1E_0^k=2(-E_0^kE_1E_0+E_0E_1E_0^k)=0$.

\item $(E_1^2-\delta E_1)E_0^kE_1E_0-E_1\left(E_1E_0^kE_1E_0-(k+1)E_1E_0\right)\\
=-\delta E_1E_0^kE_1E_0+(k+1)E_1^2E_0
=-\delta (k+1)E_1E_0+(k+1)\delta E_1E_0=0$.
\item $(E_iE_1-E_1E_i)E_0^kE_1E_0-E_i\left(E_1E_0^kE_1E_0-(k+1)E_1E_0\right)\\
    =-E_1E_iE_0^kE_1E_0+(k+1)E_iE_1E_0
=-E_i\left(E_1E_0^kE_1E_0-(k+1)E_1E_0\right)=0$.
\item $(E_1E_{i,1}-E_1E_{i,3})E_0^kE_1E_0-E_1E_{i,2}\left(E_1E_0^kE_1E_0-(k+1)E_1E_0\right)\\
=-E_1E_{i,3}E_0^kE_1E_0+(k+1)E_1E_{i,1}E_0
=-E_{i,3}\left(E_1E_0^kE_1E_0-(k+1)E_1E_0\right)=0$.
\end{enumerate}

%\medskip
%\item Take the composition between the relation $E_0E_1E_0^k-E_0^kE_1E_0$ and itself.
%\begin{enumerate}[{\rm 1)}]
%\item $(E_0E_1E_0^k-E_0^kE_1E_0)E_1E_0^k-E_0E_1E_0^{k-1}(E_0E_1E_0^k-E_0^kE_1E_0)\\
%=-E_0^kE_1E_0E_1E_0^k+E_0E_1E_0^{2k-1}E_1E_0=-2E_0^kE_1E_0^k+2E_0^{2k-1}E_1E_0\\
%=-2E_0^{k-1}(E_0E_1E_0^k-E_0^kE_1E_0)=0$.
%\end{enumerate}

\medskip
\item Finally, consider $E_{i,1}E_0^kE^{1,j}E_i-E_{i-2,1}E_0^kE^{1,j}E_i$ and another relation.
\begin{enumerate}[{\rm 1)}]
\item $(E_{i,1}E_0^kE^{1,j}E_i-E_{i-2,1}E_0^kE^{1,j}E_i)E_i-E_{i,1}E_0^kE^{1,j}(E_i^2-\delta E_i)\\
=-E_{i-2,1}E_0^kE^{1,j}E_i^2+\delta E_{i,1}E_0^kE^{1,j}E_i=\delta(-E_{i-2,1}E_0^kE^{1,j}E_i+E_{i,1}E_0^kE^{1,j}E_i)=0$.
\item $(E_{i,1}E_0^kE^{1,j}E_i-E_{i-2,1}E_0^kE^{1,j}E_i)E_j-E_{i,1}E_0^kE^{1,j}(E_iE_j-E_jE_i)\\
=-E_{i-2,1}E_0^kE^{1,j}E_iE_j+E_{i,1}E_0^kE^{1,j}E_jE_i\\
=(-E_{i-2,1}E_0^kE^{1,j}E_i+E_{i,1}E_0^kE^{1,j}E_i)E_j=0$.
\item $(E_{i,1}E_0^kE^{1,j}E_i-E_{i-2,1}E_0^kE^{1,j}E_i)E_{i-1,j}E_i-E_{i,1}E_0^kE^{1,j}(E_{i,j}E_i-E_{i-2,j}E_i)\\
=-E_{i-2,1}E_0^kE^{1,j}E_{i,j}E_i+ E_{i,1}E_0^kE^{1,j}E_{i-2,j}E_i\\
=(-E_{i-2,1}E_0^kE^{1,j}E_i+E_{i,1}E_0^kE^{1,j}E_i)E_{i-2,j}=0$.
\item $(E_{i,1}E_0^kE^{1,j}E_i-E_{i-2,1}E_0^kE^{1,j}E_i)E_{\ell,i}-E_{i,1}E_0^kE^{1,j}(E_iE_{\ell,i}-E_iE_{\ell,i+2})\\
=-E_{i-2,1}E_0^kE^{1,j}E_iE_{\ell,i}+ E_{i,1}E_0^kE^{1,j}E_iE_{\ell,i+2}\\
=(-E_{i-2,1}E_0^kE^{1,j}E_i+E_{i,1}E_0^kE^{1,j}E_i)E_{\ell,i+2}=0$.
\item $(E_jE_{i,j}-E_jE_{i,j+2})E_{j+1,1}E_0^kE^{1,j}E_i-E_j(E_{i,1}E_0^kE^{1,j}E_i-E_{i-2,1}E_0^kE^{1,j}E_i)\\
=-E_jE_{i,j+2}E_{j+1,1}E_0^kE^{1,j}E_i+E_jE_{i-2,1}E_0^kE^{1,j}E_i\\
=-E_j(E_{i,1}E_0^kE^{1,j}E_i-E_{i-2,1}E_0^kE^{1,j}E_i)=0$.
\end{enumerate}
\end{enumerate}

Hence Theorem \ref{cor-1} yields that $\widehat{R}_{\tl(d,n)}$ is a \GS basis for $\tl(d,n)$.
\end{proof}

\noindent {\it Remark}.
We can easily check that all $\widehat{R}_{\tl(d,n)}$-standard monomials are the ones in $M_{\tl(d,n)}$. %\\[.1pt]
Following the same procedure as in the proof of \cite[Theorem 4.2]{KimSSLeeDI},
 the number of monomials of the form (\ref{monomial_ctl0+}) is $$
(d-1)\sum_{s=0}^{n-1}C(n,s)d^{n-1-s}=(d-1)\mathfrak{F}_{n,n-1}(d)$$
by counting the monomials in (\ref{monomial_ctl0+}) according to $i_q(=n-1-s)=0,1,2,\ldots,n-1$ via $G(d,1,n)$-Dyck paths.

More precisely, if $i_q=0$ then the monomials are of the form $E_0^kE_{i_2,j_2}\cdots E_{i_p,j_p}$, so the number is $(d-1)C_n=(d-1)C(n,n-1)$. If
$i_q=1$ then the number of the monomials is $d(d-1)C(n,n-2)$.
If $i_q=2$ then the number is $d^2(d-1)C(n,n-3)$. In general, if $i_q=\ell\le n-1$ then the number is $d^\ell(d-1)C(n,n-1-\ell)$.

On the other hand, the number of monomials of the form (\ref{monomial_ctl0-}) is $$(d-1)(C_n-1).$$
Thus
the cardinality of the set $M_{\tl(d,n)}$ %, i.e. the set of $\widehat{R}_{\tl(d,n)}$-standard monomials,
is
$$
\begin{aligned}\dim \tl(d,n)=|M_{\tl(d,n)}|&= %C_n+ (d-1){\textstyle \frac{n+1}{2}}C_n+(d-1)({\textstyle \frac{n+1}{2}}C_n-1)
C_n +  (d-1)\mathfrak{F}_{n,n-1}(d)  +(d-1)(C_n-1)\\
&= %(dn-n+d)C_n-d+1
 (d-1)(\mathfrak{F}_{n,n-1}(d)-1)+dC_n.
\end{aligned}$$

In particular, with specialisation $d=2$, we recover the result for type $B_n$:
$$\dim \tl(B_n)=(n+2)C_n-1.$$
%\\

\begin{ex}\label{STmonoTypeA}
%Let $f_n(d)=\F_{n,n-1}(d)$. Note that $$(d-1)f_n(d)=d^2 f_{n-1}(d)+dC_{n-1}-C_n.$$
%Consider the cases of $d=3$. Simply we let $f_n=\F_{n,n-1}(3)$. Then $$
%f_1=1,\ f_2=5,\ f_3=23,\ f_4=104,\ f_5=468,\ f_6=2023, $$
%and so on. Few first numbers for $|M_{\tl(3,n)}|$ corresponding to $n=1,2,\ldots, 6$ respectively are~:
%$$3,\ 14,\ 59,\ 248,\ 1060,\ 4440,\ ... .$$
%

Consider the cases of $n=3$. The the dimension of $\tl(d,3)$ is $$
(d-1)(\mathfrak{F}_{3,2}(d)-1)+dC_n=(d-1)(d^2+3d+4)+5d=d^3+2d^2+6d-4.$$

Notice that, for $n=3$ and $d=3$, $|M_{\tl(3,3)}\setminus M_{\tl(B_3)}|=%43-24=19$.
59-24=35$.
%which is the number of standard monomials in $M_{\tl(3,3)}$ containing $E_0^2$. %If we denote by $$E_{i,j}^{(k)}=E_{i,1}E_0^kE^{1,j},$$ then
The explicit and complete list of the standard monomials in $M_{\tl(3,3)}\setminus M_{\tl(B_3)}$, that is, the standard monomials containing $E_0^2$,
is as follows:
\begin{eqnarray*}
&
E_0^2,
E_0^2E_1,
E_0^2E_1E_2,
E_0^2E_{2,1},
E_0^2E_2,\\
&
E_0^2E_{1,0},
E_0^2E_{1,0}E_{2,1},
E_0^2E_{1,0}E_2,
E_1E_0^2,
(E_1E_0^2)E_{2,1},
(E_1E_0^2)E_2,
\\
&
E_0(E_1E_0^2),E_0(E_1E_0^2)E_{2,1},E_0(E_1E_0^2)E_2, %\er\\ &\ber
E_0^2(E_1E_0^2),E_0^2(E_1E_0^2)E_{2,1},E_0^2(E_1E_0^2)E_2,\\
&E_0^2E_{2,0},
E_0^2E_{1,0}E_{2,0},
(E_1E_0^2)E_{2,0},
E_0(E_1E_0^2)E_{2,0},E_0^2(E_1E_0^2)E_{2,0},\\
&E_{2,1}E_0^2,
 E_0(E_{2,1}E_0^2),E_0^2(E_{2,1}E_0^2),
E_{1,0}(E_{2,1}E_0^2),(E_1E_0^2)(E_{2,1}E_0^2),
E_0E_{1,0}(E_{2,1}E_0^2), \\ &
E_0^2E_{1,0}(E_{2,1}E_0^2),E_0(E_1E_0^2)(E_{2,1}E_0^2),E_0^2(E_1E_0^2)(E_{2,1}E_0^2),\\
&E_1E_0^2E_1,
(E_1E_0^2E_1)E_2,
E_{2,1}E_0^2E_1,
E_{2,1}E_0^2E^{1,2}.
\end{eqnarray*}
\end{ex}

\medskip
\section{ Combinatorial aspects - connections to
Fully commutative elements and Dyck paths}\label{sectionCOMB}
%\beb
In this section, we introduce some combinatorial tools which are useful for the enumeratioin  of standard monomials induced from a {\em \GS basis} for a
Temperley-Lieb algebra of Coxeter groups. We later try to set up a bijective correspondence between the standard monomials and the fully commutative elements. %\eb
%\ber  What are the difference between these two sets representing the {\em \GS basis} for a
%Temperley-Lieb algebra of Coxeter groups?
%It would be nicer if we can give bijective correspondance between these two sets
%representing the {\em \GS basis}. I mean which are the standard elements having
%the same suffix so as to make them into which packet...etc.. \er
\smallskip

Let   $W$ be a Coxeter group. An element $w \in W$ is said to
be {\em fully commutative} if any reduced word for $w$ can be
obtained from any other  by  interchanges of adjacent commuting
generators. The  fully commutative elements play an important role particularly
for computing
the dimension of the Temperley--Lieb algebra of the Coxeter groups.\smallskip

Stembridge \cite{Stembridge96} classified all of the Coxeter
groups which have finitely many fully commutative elements. His
results completed the work of Fan \cite{Fan96}, which was done only
for the simply-laced types. In the same paper \cite{Fan96}, Fan
showed that the fully commutative elements parameterize natural
bases corresponding to certain quotients of Hecke algebras.  In type $A_n$,
these give rise to the Temperley--Lieb algebras (see
\cite{Jones1987}). Fan and Stembridge also enumerated the set of
fully commutative elements.  In particular, they showed the
following property.

\begin{prop}[\cite{Fan96,Stembridge97}]\label{FCforABD}  Let $C_n$ be the $n^\textrm{th}$ Catalan number, i.e. $C_n = \frac{1}{n+1}{2n\choose n}$.
Then the numbers of fully commutative elements in the Coxeter group
of types $A_n$, $D_n$  and $B_n$ respectively are given as follows:

$$
\begin{cases}
C_{n+1} & \text{ if the type is } A_n,\\
\frac{n+3}{2} \times C_{n}-1 & \text{ if the type is } D_n, \\
(n+2) \times  C_{n}-1  & \text{ if the type is } B_n.
\end{cases}
$$
\end{prop}
%\smallskip

\noindent {\it Remark}.
There are several combinatorial ways to realize the Catalan number $C_n$. Among those, we consider the Dyck paths in $n \times n$-lattice plane starting from the point at $(0,0)$  and ending at the point at $(n,n)$ using only with northern and eastern directions at each step and never passing above the diagonal line. Seeing that
the dimension of the Temperley-Lieb algebra $\tl(A_{n-1})$ of the Coxeter group of type $A_{n-1}$
is the Catalan number $C_n$ which is the number of fully commutative elements,
we can realize a bijective correspondence between the Dyck paths and the fully
commutative elements for $\tl(A_{n-1})$ as we can see in the article \cite{FeinLee2}.

\medskip

The number of fully commutative elements in Coxeter groups of type $B$ and $D$ respectively in Proposition \ref{FCforABD}  above can also be understood in pure combinatorial way by using the notions of
Catalan's triangle and Packets as mentioned in the articles \cite{FeinLee} and \cite{FKLO}. In this section, we recall the definitions and some combinatorial properties of Catalan numbers and Catalan triangles as well as the main results in \cite{FeinLee} and \cite{FKLO}. %\\[.1pt]
\smallskip

The Catalan numbers can be defined in a recursive way as follows~: we set the first entry $C(0,0)=1$. For $n\geq0$ and $0\leq k \leq n$, we denote $C(n,k)$  ($0 \leq k \leq n$)  the entry in the $n^{\textrm{th}}$ row and $k^\textrm{th}$ column of Table \ref{CT} below, called the {\em Catalan's Triangle}, verifying the following recursive formula.
  \begin{equation*}\label{sum rule}
    C(n,k) =
    \left\{
    \begin{array}{cl}
        1 & \textrm{ if } n=0 ;\\
        C(n,k-1)+C(n-1,k) & \textrm{ if } 0<k<n; \\
        C(n-1, 0) & \textrm{ if } k=0; \\
        C(n, n-1)=C(n,n)=C_n & \textrm{ if } k=n.\\
        0 & \textrm{ if } k>n.\\
    \end{array}
    \right.
  \end{equation*}

\begin{table}[h]
\[\begin{array}{ccccccccc}
1&\\
1 & 1(=C_1)\\
1 & 2 & 2(=C_2)\\
1 & 3 & 5 & 5(=C_3)\\
1 & 4 & 9 & 14 & 14 (=C_4)\\
1 & 5 & 14 & 28 & 42 & 42 (=C_5)\\
%1 & 6 & 20 & 48 & 90 & 132 & 132\\
%1 & 7 & 27 & 75 & 165 & 297 & 429 & 429\\
\vdots & \vdots &\vdots &\vdots &\vdots &\vdots &\vdots & \vdots& \ddots\\
\end{array}
\hskip-1.5cm  C(n,k) = \frac{(n+k)!(n-k+1)}{k!(n+1)!}
\]
\caption{Catalan's Triangle}\label{CT}
\end{table}

%\begin{table}[h]
%\[\begin{array}{ccccccccc}
%1\\
%1 & 1\\
%1 & 2 & 2\\
%1 & 3 & 5 & 5\\
%1 & 4 & 9 & 14 & 14\\
%1 & 5 & 14 & 28 & 42 & 42\\
%1 & 6 & 20 & 48 & 90 & 132 & 132\\
%1 & 7 & 27 & 75 & 165 & 297 & 429 & 429\\
%\vdots & \vdots &\vdots &\vdots &\vdots &\vdots &\vdots & \vdots& \ddots\\
%
%\end{array}
%\]
%\caption{Catalan's Triangle} \label{CT}
%\end{table}

\smallskip

%\begin{Remark}
\noindent {\it Remark}. %\beb
We consider the Dyck paths in $n \times n$-lattice plane starting from the point at $(0,0)$  and ending at the point at $(n,n)$ using only with northern and eastern directions at each step and never passing above the diagonal line. Then the coefficient $C(n,k)$ in the Catalan's triangle can be understood as the number of these Dyck paths passing through exactly the point $(n,k)$ ($k \leq n$). %\eb
%\end{remark}
\medskip

%\subsection{Fully commutative element}
%The main theorem in \cite{KleRam} shows that the homogeneous components of a weight graph parametrize the homogeneous representations of the corresponding KLR algebra. In \cite{FeinLee2}, they showed that the homogeneous components are characterized by the fully commutative elements written in canonical form. The
%Figure~\ref{DyckPathsA} shows the homogeneous components and their
%corresponding fully commutative elements for type $A_2$.

\medskip

 In \cite{FeinLee}, Feinberg and Lee showed that the set of fully commutative elements can be decomposed into subsets called {\it collections} depending on the shape of the suffix of each reduced element written in canonical form. Then they
 proved that some collections have the same set of prefixes, i.e. the same cardinality.
 Now, we group those collections together and call the group a {\em packet} and
we call this process the {\it packet decomposition}.

\medskip

Following the notations in (\ref{BnDynkin}) and (\ref{Eq: WB}), for $0\leq i < n$, we define the words $s_{ij}$ and $s^{j,i}$ by:
$$ s_{ij} =
 \left\{
   \begin{array}{ll}
     ~[i, i-1, \hdots, j] := s_i \cdots s_j& \text{ if } i > j, \\
     ~[i] : =s_i & \text{ if } i = j, \\
     ~[~] :=1& \textrm{ if } i<j .
   \end{array}
 \right.
 \quad
  s^{ji} =
 \left\{
   \begin{array}{ll}
     ~s_j s_{j+1}\cdots s_i& \text{ if } i>j, \\
     ~[i] : =s_i & \text{ if } i = j, \\
     ~[~] :=1 & \textrm{ if } i<j .
   \end{array}
 \right.
$$

\begin{lem}{\rm (\cite[Proposition 2.3]{KLLO1}, %Canonical Form,
cf. \cite[Lemma 4.2]{BokutShiao})}\label{Bcan}
Any element of the Coxeter group of type $B_n$ can be uniquely
written in the reduced form
$$ %\[
   s_{0,a_0}^{(k_0)}s_{1,a_1}^{(k_1)}s_{2,a_2}^{(k_2)}\cdots s_{n-1,a_{n-1}}^{(k_{n-1})}
 %\]
 $$
 where
$$1 \le a_i \leq i+1, \,\,k_i=0 \mbox{ or } 1,\,\,\, \textrm{and } s_{ij}^{k} = \begin{cases}
s_{ij} & \textrm{ if } k = 0,\\
s_{i,0} s^{1,j-1}&\textrm{ if }  k=1.
\end{cases}$$
\end{lem}

\noindent {\it Remark}. Notice that
there are $i+1$ choices for each $a_i$
since $1\le a_i\leq i+1$ while $i=0,1,\ldots,n-1$, thus $n!$ choices
for the values of $a_i$'s altogether. There are $2$ ways for each exponent $k_i$ of $a_i$,
and thus there are $2^n$ choices for the exponents.
Thus we have $n!\cdot 2^{n}$ elements in the canonical reduced form.
We recall that there are the same number of elements in the type $B_n$ Coxeter group.
\medskip

Among the above canonical elements, since, for $1\le i_1<i_2$ and $j_1>1$,
 any element $$\begin{aligned}
s_{i_1,j_1}^{(1)}s_{i_2,j_2}^{(1)}=s_{i_1,0} s^{1,j_1-1}s_{i_2,0} s^{1,j_2-1}
=s_{i_1,0} s^{1,j_1-2}s_{i_2,j_1+1}(s_{j_1-1}s_{j_1}s_{j_1-1})s_{j_1-2,0} s^{1,j_2-1}
\end{aligned}$$
contains a braid relation $s_{j-1}s_js_{j-1}$, thus any element containing this form is not fully commutative.
Also,  for $j>1$, any element containing the form
$$s_0s_{i,j}^{(1)}=s_0s_{i,0}s^{1,j-1}=s_{i,2}(s_0s_1s_0s_1)s^{2,j-1}$$ is non-fully-commutative element.
%not fully commutative.
\medskip

Now we consider only the elements of the form \begin{equation} \label{B_fc1}
   s_{i_1,0}s_{i_2,0}\cdots s_{i_\ell, 0}s_{i_\ell+1,a_{i_\ell+1}}\cdots s_{n-1,a_{n-1}}
\end{equation} where $0\leq i_1<i_2<\cdots<i_\ell < n$ for $\ell\geq0$, or
\begin{equation} \label{B_fc2}
%s_{i,a_i}^{(1)}
(s_{i,0}s^{1,a_i})s_{i+1,a_{i+1}}\cdots s_{n-1,a_{n-1}}
\end{equation} with $i\ge 1$.

The left factors $s_{i_1,0}s_{i_2,0}\cdots s_{i_\ell, 0}$ in (\ref{B_fc1}),
and $s_{i,0}$ in (\ref{B_fc2}) will be
called the \textit{prefix}.
Similarly the right factors
%$s_{j_1,1}s_{j_2,2}s_{j_3,3}\cdots s_{j_{n-1},n-1}$
$s_{i_\ell+1,a_{i_\ell+1}}\cdots s_{n-1,a_{n-1}}$ in (\ref{B_fc1}), and
$s^{1,a_i}s_{i+1,a_{i+1}}\cdots s_{n-1,a_{n-1}}$ in (\ref{B_fc2})
will be called
the \textit{suffix} of the reduced word. Given a reduced word $\bw$
in the extracted canonical form, we will denote by $\pw$ the prefix and
by $\sw$ the suffix of $\bw$ and  we can write, in a unique way, $\bw = \pw \sw$.
%Generally, a word of the form $s_{i_1,0}s_{i_2,0}\cdots s_{i_\ell, 0}$ with $0\leq i_1<i_2<\cdots<i_\ell < n$ for $\ell \ge 0$ will be called a {\em prefix}.
\medskip
%\noindent {\it Remark}. Notice that choosing a prefix is equivalent to choosing a (possibly empty) subset of $\{1, 2, \hdots, n\}$.
%There are $2^{n}$ ways to do this.  Since there are $k$ choices for each $j_k$  ($k=1,\ldots, n-1$),
%there are $n!$ choices of suffixes. Thus we have $n!\cdot 2^{n}$ elements in the canonical reduced form. We recall that there are the
%same number of elements in the type $B_n$ Coxeter group.

%As is proven in \cite{KimLeeOh}, we can prove
We notice that every prefix is a fully commutative element.
%We denote $\mathcal C_{w_p}$ the set of fully commutative elements in $W(B_n)$.
Some of the collections in the set of {\em fully commutative} elements of $W(B_n)$
have the same number of elements and thus we
group them together into a set called a {\it $(n,k)$-packet} depending on the form
of the prefixes as follows~:

%\begin{lem} \label{num}
%The collections in a packet have the same number of elements, which is $C(n,k)$.
%\end{lem}
%
%\begin{remark}
%When we consider the Dyck path as a lattice path in $n \times n$-lattice plane which does not pass above the diagonal,
%$C(n,k)$ can be understood as the number of Dyck paths passing through the vertex $(n,k)$.
%\end{remark}

\begin{defn}
For $0 \le k \le n$, we define the {\em $(n,k)$-packet} of
collections:
\begin{itemize}
\item The $(n,0)$-packet is the set of collections labeled by prefixes of the form \[s_{i_1,0}s_{i_2,0}\cdots s_{i_\ell, 0}s_{n-1,0}\quad (\ell \ge 1).\]
\item The $(n,k)$-packet, $1\le k \le n-2$, is the set of collections labeled by $s_{n-k,0}$  or
prefixes of the form
\[s_{i_1,0}s_{i_2,0}\cdots s_{i_\ell, 0}s_{n-k-1,0}\quad  (\ell  \ge 1).\]
\item The $(n,n-1)$-packet contains only the collection labeled by $s_0=[0]$ or $s_{1,0}=[1,0]$.
\item The $(n,n)$-packet contains only the collection labeled by the empty prefix $[~]$.
\end{itemize}

\smallskip

We will denote the $(n,k)$-packet by $\PP_B(n,k)$. As an example, Table~\ref{BijSTFC} shows all of collections of the packets for type $B_3$.
\end{defn}

\begin{prop}[\cite{FKLO}]
The cardinality of the packet $\mathcal P_B(n,k)$ for type $B_n$ is
$$ \left |\mathcal P_B(n,k) \right| =
   \left\{\begin{array}{cl}
   2^{n-1}-1 & \textrm{ if } k = 0, \\
   2^{n-k-1} & \textrm{ if } 1\le k \le n-2, \\
   2 & \textrm{ if } k= n-1, \\
   1 & \textrm{ if } k= n. \\
   \end{array}
   \right.
$$
Therefore $ \sum_{k=0}^n  \left |\mathcal P_B(n,k) \right|  =  2^{n}$.
We remark that for a fixed $k$ with $0\leq k\leq n$ \rm ($n\geq 3$\rm), each collection with same prefix in the
packet $\mathcal P_B(n,k)$ has the same cardinality which is $C(n,k)$.
\end{prop}

%We refer the paper \cite{KimLeeOh} for the detailed proofs of the following properties~:
%
%\begin{thm} \label{Collections}
%Assume that $n\ge 3$ and $0\le k \le n$. Then we have:
%\begin{enumerate}[{\rm (a)}]
%\item Each collection in the packet $\PP(n,k)$ contains exactly $C(n,k)$   elements.
%\item Each collection in the packets \beb $\PP(n,n-1)$ and $\PP(n,n)$  contains $C_n$  fully commutative elements. \eb %words in $\mathcal W_n$.
%\end{enumerate}
%\end{thm}

\begin{cor} %\label{cor-end}
For $n\ge 3$, we obtain the identity:
\begin{equation} \label{eqn-end_B} \sum_{k=0}^n C(n,k) \left | \PP_B(n, k) \right | = (n+2) \times C_n -1 .\end{equation}
\end{cor}
\medskip

We generalize the above method for the complex reflection group of type $G(d,1,n)$.

%\medskip

\begin{lem}{\rm (Canonical Form for $G(d,1,n)$ \cite[Proposition 2.3]{KLLO1})}\label{Bcan}
Any element of the Coxeter group of type $G(d,1,n)$ can be uniquely
written in the reduced form
 \[
   s_{0,a_0}^{(k_0)}s_{1,a_1}^{(k_1)}s_{2,a_2}^{(k_2)}\cdots s_{n-1,a_{n-1}}^{(k_{n-1})}
 \]
 where
$$1 \le a_i \leq i+1, \,\,0 \le k_i \le d-1,\,\,\, \textrm{and } s_{ij}^{k} = \begin{cases}
s_{ij} & \textrm{ if } k = 0,\\
s_{i,1}s_0^k s^{1,j-1}&\textrm{ if }  k=1,2,\ldots,d-1.
\end{cases}$$
\end{lem}

\medskip
\noindent {\it Remark}. Notice that
there are $n!$ choices of $a_i$'s ($i=0,1,\ldots, n-1$), and
there are $d^{n}$ ways to choose $k_i$'s.
%there are $i+1$ choices for each $a_i$
%since $1\le a_i\leq i+1$ while $i=0,1,\ldots,n-1$, thus $n!$ choices
%for the values of $a_i$'s altogether. There are $d$ ways for each exponent $k_i$ of $a_i$,
%and thus there are $d^n$ choices for the exponents.
Therefore we have $n!\cdot d^{n}$ elements in the canonical reduced form.
We recall that there are the same number of elements in the complex reflection group of type
$G(d,1,n)$.\\

%The followings are the conjectures on the fully commutative elements of the complex reflection group of type $G(d,1,n)$.\medskip

%\noindent {\bf Conjecture 1~:} None of the words containing $s_0s_1s_0^2$ is fully commutative.\\

%If Conjecture 1 is true, then we have the following proposition. \medskip

%\noindent {\bf Conjecture 2~:}
The fully commutative elements of the complex reflection group of type $G(d,1,n)$ are of the form
\begin{equation}\label{Gd1n_fc1}
s_{i_1,1}s_0^{k_1} s_{i_2,1}s_0^{k_2}\cdots s_{i_q,1}s_0^{k_q}s_{i_{q+1},j_{q+1}}\cdots s_{i_p,j_p}\end{equation} with
$$\begin{aligned} &0\le p\le n,\quad 0\le q\le p,\quad  1\le k_1,k_2,\ldots,k_q<d,\quad \\
&0\le i_1<i_2<\cdots <i_p\le n-1, \quad 0< j_{q+1}< \cdots \le j_p\le n-1,\ \mbox{ and }\\
&i_{q+1}\ge j_{q+1},\  \ldots,\ i_p\ge j_p,\end{aligned}$$ or
\begin{equation}\label{Gd1n_fc2}(s_{i_1,1}s_0^ks^{1,j_1})s_{i_2,j_2}s_{i_3,j_3}\cdots s_{i_p,j_p}
\end{equation} with %the same restriction on $i$'s and $j$'s as above.
$$\begin{aligned} &1\le p\le n-1,\quad  1\le k<d, \\
&1\le i_1<i_2<i_3< \cdots <i_p\le n-1,\quad 1\le j_1<j_2<j_3<\cdots <j_p\le n-1,\ \mbox{ and }\\
&i_1\ge j_1,\ i_2\ge j_2,\ \ldots,\ i_p\ge j_p.\end{aligned}$$

The left factor $s_{i_1,1}s_0^{k_1} s_{i_2,1}s_0^{k_2}\cdots s_{i_q,1}s_0^{k_q}$ with $q\ge 0$ in (\ref{Gd1n_fc1}) and (\ref{Gd1n_fc2}) will be
called the \textit{prefix} of the fully commutative element of type $G(d,1,n)$.

%, and similarly the right factor $s_{j_1,1}s_{j_2,2}s_{j_3,3}\cdots s_{j_{n-1},n-1}$ will be called
%the \textit{suffix} of the reduced word. Given a reduced word $\bw$ in canonical form, we will denote by $\pw$ the prefix of $\bw$ and
%by $\sw$ the suffix, and write $\bw = \pw \sw$.
%Generally, a word of the form $s_{i_1,1}s_0^ms_{i_2,0}\cdots s_{i_\ell, 0}s_{n-k-1,0}$ with
%$0\leq i_1<i_2<\cdots<i_\ell < n (??)$ for $\ell \ge 0$ will be called a {\em prefix}.

%\medskip

\begin{defn}
 For $0 \le s \le n$% and $1\le m <d$
, we define the {\em $(n,s)$-packet} of collections:
\begin{itemize}
\item The $(n,0)$-packet is the set of collections labeled by prefixes of the form \[s_{i_1,1}s_0^{k_1}s_{i_2,1}s_0^{k_2}\cdots s_{i_{q-1}, 1}s_0^{k_{q-1}}s_{n-1,1}s_0^{k_q}\quad (q \ge 2).\]
\item The $(n,s)$-packet, $1\le s \le n-2$, is the set of collections labeled by %\ber
$s_{n-s,1}s_0^k$ %\er
or prefixes of the form
\[s_{i_1,1}s_0^{k_1}s_{i_2,1}s_0^{k_2}\cdots s_{i_{q-1}, 1}^{k_{q-1}}s_{n-s-1,1}s_0^{k_q}\quad  (q  \ge 2).\]
\item The $(n,n-1)$-packet contains only the collections labeled by $s_0^k=[0^k]$ or $s_1s_0^k=[1,0^k]$.
\item The $(n,n)$-packet contains only the collection labeled by the empty prefix $[~]$.
These elements cover the elements of $\mathfrak S_{n} \subset G(d,1,n)$.
\end{itemize}

%\smallskip

We will denote the $(n,s)$-packet by $\PP(n,s)$. %In particular, $\PP(n, k, 1)=\PP_B(n, k)$.
 For clarifying the packet decomposition,
we denote $(n,s)$-packet of type $G(d,1,n)$ which is not contained in the $(n,s)$-packet of type $B_n$
by $$\PP(n,s)'=\PP(n,s)\setminus \PP_B(n,s).$$
\end{defn}

\begin{prop}{\rm (cf. \cite[\S4]{FKLO})}
\begin{enumerate}[{\rm (a)}]
\item Every collection in the packet $\PP(n,s)$ has $C(n,s)$ elements.
\item The size of the packet $\PP(n,s)$ is $$.
|\PP(n,s)|=\left\{
  \begin{array}{ll}
    (d-1)(d^{n-1}-1) & \hbox{if } s=0,\\
    (d-1)d^{n-1-s} & \hbox{if } 1\le s\le n-2,\\
    2(d-1) & \hbox{if } s=n-1,\\
    1 & \hbox{if } s=n.
  \end{array}
\right.
$$
\end{enumerate}
\end{prop}

\begin{proof}
 (a) The collections in $\PP(n,0)$ have no other element than the prefix. It's also clear that
the collections in $\PP(n,n-1)$ have $C_n=C(n,n-1)$ elements and the collection in $\PP(n,n)$ has $C_n=C(n,n)$ elements.

For the collections in $\PP(n,s)$ with $1\le s\le n-2$, the number of suffixes attached to the given prefix is exactly $C(n,s)$
by counting the Dyck paths from $(n-s,0)$ to $(n,n)$ via the Dyck paths, reflected relative to $y=n-x$ in the $xy$-plane, from $(0,0)$ to $(n,s)$.

(b) Counting the number of prefixes in each packet, we get its packet size.
\end{proof}

In the same way as the formula (\ref{eqn-end_B}), we obtain the identity for the complex reflection group of type $G(d,1,n)$:
$$\begin{aligned}
%C_n+\sum_{m=1}^{d-1}\sum_{k=0}^{n-1} C(n,k) \left | \PP(n, k,m) \right | = C_n +\sum_{k=0}^{n-1} C(n,k) \left | \PP(n, k) \right |(d-1)
\sum_{s=0}^{n} C(n,s) \left | \PP(n, s) \right |&= (d-1)(d^{n-1}-1)+\sum_{s=1}^{n-2}C(n,s)(d-1)d^{n-1-s}+2(d-1)C_n+C_n \\
&= %C_n+(d-1)\sum_{s=0}^{n-1}C(n,s)d^{n-1-s}+(d-1)(C_n-1)=
(d-1)\left(\mathfrak{F}_{n,n-1}(d)-1\right)+dC_n,\end{aligned}$$
which is equal to $\dim \tl(d,n)$.\\
%\ber We now try to generalize the above method for the complex reflection group of type
%$G(d,1,n)$. For this, we need to generalize the result of Stembridge(?) so that we can
%decompose the set of fully commutative elements of $G(d,1,n)$, but there is an obstacle for this.
%
%For the case of Coxeter groups, any two reduced expressions of an element can be
%switched (or transferred)  from one to another by only braid relations. But this property
%does not work for the case of $G(d,1,n)$~: considering the group
%$G(3,1,3) = \langle a,b\,|\, a^3=b^2=1,abab=baba\rangle$,
%we can see that $baaba$ and $abaab$ are two reduced expressions representing the
%same word. But there is no way to transform from one to another by braid relations.\er
%
%{\bf Conjecture

As an example, Table~\ref{BijSTFC} shows all of collections of the packets for type $G(3,1,3)$.

{\scriptsize
\begin{table}
\begin{tabular}{|c|c|c|}
  \hline \Tstrut
   & \text{\bf Standard Monomials} & \text{\bf Fully Commutative Elements} \Bstrut\\
   \hline \Tstrut
  $M_{\tl(A_2)}$ & $\text{1}, E_1, E_2, E_{2,1}=E_2E_1,E_1E_2$  &$\mathcal P(3,3)= \mathcal C_{[~]}=\{ id, s_1, s_2, s_2s_1,s_1s_2 \}$
\Bstrut\\
   \hline \Tstrut
  $M^0_{\tl(B_3)}$& $\begin{cases}E_0,E_0E_2,E_0E_1E_2,E_0E_1,E_0E_2E_1, \\
E_{1,0},E_{1,0}E_2,E_{1,0}E_1E_2, E_{1,0}E_1,E_{1,0}E_{2,1} \end{cases}$ & %\beb\mathcal P(3,2)=\eb\newline
 $\mathcal P_B(3,2)= \begin{cases} \mathcal C_{[0]}=\{{\bf s_0},{\bf s_0}s_2,{\bf s_0}s_1s_2,{\bf s_0}s_1,{\bf s_0}s_2s_1 \} \\
  \mathcal C_{[1,0]}=\{{\bf
  s_1s_0},{\bf s_1s_0}s_2, {\bf s_1s_0}s_1s_2, {\bf s_1s_0}s_1, {\bf s_1s_0}s_2s_1\}
  \end{cases}$
 \\
\rule{0pt}{7ex}
& $\begin{cases}E_{2,0}, E_{2,0}E_1,E_{2,0}E^{1,2},\\
E_0E_{1,0},E_0E_{1,0}E_2, E_0E_{1,0}E_{2,1}\end{cases}$ & $\mathcal P_B(3,1)=\begin{cases}   \mathcal C_{[2,1,0]}=\{
{\bf %s_{2,0}=
s_2s_1s_0},{\bf s_2s_1s_0}s_1, {\bf s_2s_1s_0}s_1s_2\} \\
\mathcal C_{[0,1,0]}=\{{\bf s_0s_1s_0},{\bf s_0s_1s_0}s_2, {\bf s_0s_1s_0}s_2s_1\} \end{cases}$
\\
\rule{0pt}{8ex}
&$\begin{cases}  E_0E_{2,0}, \\
    E_{1,0}E_{2,0}, \\
    E_0E_{1,0}E_{2,0} \quad \text{ Monomials containing}\,\, E_0
          \end{cases}$
        & $\mathcal P_B(3,0)=\begin{cases}  \mathcal C_{[0,2,1,0]}=\{{\bf s_0s_2s_1s_0}\} \\
   \mathcal C_{[1,0,2,1,0]}=\{{\bf s_1s_0s_2s_1s_0}\} \\
      \mathcal C_{[0,1,0,2,1,0]}=\{{\bf s_0s_1s_0s_2s_1s_0}\}
      \end{cases}$\\
%\vspace{5pt}
& &
%\rule{0pt}{7ex}
%\Tstrut\Bstrut
%\rule{0pt}{1.5\normalbaselineskip}
\\
   \hline \Tstrut\Bstrut
    $M^{0^2}_{\tl(3,3)}$ &  $\begin{cases}E_0^2,E_0^2E_2,E_0^2E^{1,2},E_0^2E_1,E_0^2E_{2,1}, \\
E_1E_0^2,E_1E_0^2E_2,E_1E_0^2E^{1,2}, E_1E_0^2E_1,
E_1E_0^2E_{2,1} \end{cases} $&
$\mathcal P(3,2)'= \begin{cases} \mathcal C_{[0^2]}= \{{\bf s_0^2},{\bf s_0^2}s_2,{\bf s_0^2}s_1s_2,{\bf s_0^2}s_1,{\bf s_0^2}s_2s_1\} \\
 \mathcal C_{[1,0^2]}=\{{\bf s_1s_0^2},{\bf s_1s_0^2}s_2,{\bf s_1s_0^2}s_1s_2, {\bf s_1s_0^2}s_1,
{\bf s_1s_0^2}s_2s_1\} \end{cases}$
\\
\rule{0pt}{7ex}
& $\begin{cases} E_{2,1}E_0^2, E_{2,1}E_0^2E_1,E_{2,1}E_0^2E^{1,2},\\
E_0^2E_{1,0},E_0^2E_{1,0}E_2, E_0^2E_{1,0}E_{2,1},\\
 E_0E_1E_0^2, E_0E_1E_0^2E_2, E_0E_1E_0^2E_{2,1},\\
 E_0^2E_1E_0^2, E_0^2E_1E_0^2E_2, E_0^2E_1E_0^2E_{2,1}\end{cases}$
&    $\mathcal P(3,1)'= \begin{cases}   \mathcal C_{[2,1,0^2]}=\{{\bf s_2s_1s_0^2},{\bf s_2s_1s_0^2}s_1, {\bf s_2s_1s_0^2}s_1s_2\} \\
\mathcal C_{[0^2,1,0]}=\{{\bf s_0^2s_1s_0},{\bf s_0^2s_1s_0}s_2, {\bf s_0^2s_1s_0}s_2s_1\} \\
 \mathcal C_{[0,1,0^2]}=\{{\bf s_0s_1s_0^2},{\bf s_0s_1s_0^2}s_2, {\bf s_0s_1s_0^2}s_2s_1\} \\
 \mathcal C_{[0^2,1,0^2]}=\{{\bf s_0^2s_1s_0^2},{\bf s_0^2s_1s_0^2}s_2, {\bf s_0^2s_1s_0^2}s_2s_1\} \end{cases}$
\\
\rule{0pt}{8ex}
& $\begin{cases}  E_0^2E_{2,0}, \\
    E_1E_0^2E_{2,0}, \\
    E_0^2E_{1,0}E_{2,0}, \\
    E_0(E_1E_0^2)E_{2,0}, \\
    E_0^2(E_1E_0^2)E_{2,0}, \\
    E_0(E_{2,1}E_0^2),\\E_0^2(E_{2,1}E_0^2),\\
E_{1,0}(E_{2,1}E_0^2),\\(E_1E_0^2)(E_{2,1}E_0^2),\\
E_0E_{1,0}(E_{2,1}E_0^2),\\
E_0^2E_{1,0}(E_{2,1}E_0^2),\\
E_0(E_1E_0^2)(E_{2,1}E_0^2),\\
E_0^2(E_1E_0^2)(E_{2,1}E_0^2) \quad \text{ Monomials containing}\,\, E_0^2
          \end{cases}$   &   $\mathcal P(3,0)'=\begin{cases}  \mathcal C_{[0^2,2,1,0]}=\{{\bf s_0^2s_2s_1s_0}\} \\
   \mathcal C_{[1,0^2,2,1,0]}=\{{\bf s_1s_0^2s_2s_1s_0}\} \\
      \mathcal C_{[0^2,1,0,2,1,0]}=\{{\bf s_0^2s_1s_0s_2s_1s_0}\}\\
 \mathcal C_{[0,1,0^2,2,1,0]}=\{{\bf s_0s_1s_0^2s_2s_1s_0}\} \\
 \mathcal C_{[0^2,1,0^2,2,1,0]}=\{{\bf s_0^2s_1s_0^2s_2s_1s_0}\} \\
 \mathcal C_{[0,2,1,0^2]}=\{{\bf s_0s_2s_1s_0^2}\} \\
 \mathcal C_{[0^2,2,1,0^2]}=\{{\bf s_0^2s_2s_1s_0^2}\} \\
 \mathcal C_{[1,0,2,1,0^2]}=\{{\bf s_1s_0s_2s_1s_0^2}\} \\
 \mathcal C_{[1,0^2,2,1,0^2]}=\{{\bf s_1s_0^2s_2s_1s_0^2}\} \\
 \mathcal C_{[0,1,0,2,1,0^2]}=\{{\bf s_0s_1s_0s_2s_1s_0^2}\} \\
 \mathcal C_{[0^2,1,0,2,1,0^2]}=\{{\bf s_0^2s_1s_0s_2s_1s_0^2}\} \\
 \mathcal C_{[0,1,0^2,2,1,0^2]}=\{{\bf s_0s_1s_0^2s_2s_1s_0^2}\} \\
 \mathcal C_{[0^2,1,0^2,2,1,0^2]}=\{{\bf s_0^2s_1s_0^2s_2s_1s_0^2}\}
      \end{cases}  $
\\
& &\\
   \hline
\end{tabular}
\medskip
\caption{\footnotesize Bijection between Standard Monomials and Fully Commutative Elements of $G(3,1,3)$}\label{BijSTFC}
\end{table}
}

\section{Temperley-Lieb algebras of types $G(d,d,n)$ and $G(d,r,n)$}

Since we obtain a $\C$-basis of $\tl(d,n)$ and its multiplication structure,
we define the subalgebras $\tl(d,d,n)$ and $\tl(d,r,n)$ in general as follows.

\begin{defn} \
\begin{enumerate}[{}\rm (a)]
\item The subalgebra $\tl(d,d,n)$ is a subalgebra of $\tl(d,n)$, whose $\C$-basis consists of
$$
\begin{cases}
\textrm{the elements  in (\ref{monomial_tlA})},\\
\textrm{the elements in (\ref{monomial_ctl0+}) satisfying}\ \
k_1+\cdots+k_q \equiv 0 \pmod{d}, \\
\textrm{the elements  in (\ref{monomial_ctl0-}) with}\ \ k=1.
\end{cases}
$$
We remark that
$$
\begin{aligned}
\dim \tl(d,d,n)& = C_n+(d-1)\mathfrak F_{n,n-2}(d)+C_n-1\\
& =(d-1) \mathfrak F_{n,n-2}(d)+2C_n-1.
\end{aligned}
$$
In particular, for $d=2$, we recover the same formula as we had in \cite{KimSSLeeDI}~:
$$\dim \tl(2,2,n)= \dim \tl(D_n)= \frac{n+3}{2}C_n-1.$$

\item More generally, %for $\tl(d,r,n)$ with $r|d$ and $r\geq 2$,
the subalgebra $\tl(d,r,n)$ ($r|d$ and $r\geq 2$) is a subalgebra of $\tl(d,n)$, whose $\C$-basis consists of
$$
\begin{cases}
\textrm{the elements  in (\ref{monomial_tlA})},\\
\textrm{the elements in (\ref{monomial_ctl0+}) satisfying}\ \
k_1+\cdots+k_q \equiv 0 \pmod{r}, \\
\textrm{the elements  in (\ref{monomial_ctl0-}) with}\ \ k\in \{1, r, 2r,\cdots,d-2r,d-r\}.
\end{cases}
$$
We remark that
$$
\begin{aligned}
\dim \tl(d,r,n)& = C_n+{ d\over r}(d-1)\mathfrak F_{n,n-2}(d)+{ d\over r}(C_n-1)\\
& ={ d\over r}(d-1) \mathfrak F_{n,n-2}(d)+\left(1+{ d\over r}\right)C_n-{ d\over r}.
\end{aligned}
$$
\end{enumerate}
\end{defn}

%\er
%\bigskip
%\newpage
\bigskip\vskip 10mm
%\proof[Acknowledgements]

%\vskip 10mm
%\bibliographystyle{amsalpha}

\end{document}